\documentclass{amsart}

\usepackage{graphicx}
\usepackage{color}
\usepackage{hyperref}
\usepackage{amssymb}
\usepackage{tikz}
\usepackage{multirow}
\usepackage{macros}

\numberwithin{equation}{section}
\numberwithin{theorem}{section}

\begin{document}

\title{Distance to  the discriminant}
\date{\today}

\author{C. Raffalli}
\address{LAMA, UMR 5127, Universit\'e de Savoie}
\email{christophe.raffalli@univ-savoie.fr}
\thanks{We wish to thank for the fruitful discussions and their
  advises our colleagues:
  Fr\'ed\'eric Bihan, Erwan Brugall\'e,
  Krzysztof Kurdyka, Olivier Le Gal, Fr\'ed\'eric Mangolte, Michel
  Raibaut, Jo\"el Rouyer.}

\maketitle

\section{Abstract}
We will study algebraic hyper-surfaces on the real unit sphere  $\mathcal
 S^{n-1}$ given by an homogeneous polynomial of degree d in n
variables with the view point, rarely exploited, of Euclidian geometry
using Bombieri's scalar product and norm. This view point is mostly
present in works about
the topology of random hyper-surfaces \cite{ShubSmale93,
  GayetWelschinger14}.

Our first result (lemma \ref{distgen} page \pageref{distgen}) is a
formula for the distance $\dist(P,\Delta)$ of a polynomial to the {\em real discriminant} $\Delta$, i.e. the set of polynomials
with a real singularity on the sphere. This formula is given for any distance coming from a
scalar product on the vector space of polynomials.

Then, we concentrate on Bombieri scalar product and its remarkable
properties. For instance we establish a
combinatoric formula for the scalar product of two products of
linear-forms (lemma \ref{scalarlinear} page \pageref{scalarlinear}) which allows to give a (new ?) proof of the invariance of
Bombieri's norm by composition with the orthogonal group.  These
properties yield a simple formula for the distance in theorem
\ref{bombineqdist} page \pageref{bombineqdist} from which we deduce
the following inequality:

$$\dist(P, \Delta) \leq \min_{x \hbox{
    critical point of } P \hbox{ on } \mathcal S^{n-1}} |P(x)|$$

The definition \ref{maindef} page \pageref{maindef} classifies in two
categories the ways to make a polynomial singular to realise the
distance to the discriminant. Then, we show, in theorem \ref{extremal} page
\pageref{extremal}, that one of the category is forbidden in
the case of an {\em extremal} hyper-surfaces (i.e. with maximal
Betti numbers). This implies as a corollary \ref{bombeqdistbis} (page \pageref{bombeqdistbis})
that the above inequality becomes an equality is that case.

The main result in this paper concerns extremal hyper-surfaces $P=0$ that maximise the
distance to the discriminant (with $\|P\| = 1$). They are very remarkable objects which enjoy
properties similar to those of quadratic forms: they are linear
combination of powers of linear forms $x \mapsto \langle x | u_i \rangle^d$
 where the vectors $u_i$ are the critical points of $P$ on $\mathcal
 S^{n-1}$ corresponding to the least positive critical value of $|P|$.
This is corollary  \ref{comblinbis} page \pageref{comblinbis} of a
similar theorem \ref{comblin} page \pageref{comblin} for all algebraic hyper-surfaces.

The next section is devoted to homogeneous polynomials in $2$
variables. We prove that a polynomial of degree $d$ with $2d$
regularly spaced roots on the unit circle is a local maximum of the
distance to the discriminant among polynomials with the same norm and
number of roots. We conjecture that it is a global maximum and that the
polynomial   of degree $d$ with $2r$
regularly spaced roots on the unit circle is also a similar global
maximum when $d < r \leq 2d$. This claim is supported by the fact that
we were able to prove the consequence of this together with corollary
\ref{comblinbis}
which yields to interesting trigonometric identities that we could not
find somewhere else (proposition \ref{trigo} page \pageref{trigo}).

We also obtain metric information about algebraic hyper-surfaces. First, in
the case of extremal hyper-surface, we
give an upper bound (theorem \ref{bandwidth} page
\pageref{bandwidth}) on the length of an integral curve of the
gradient of $P$ in the band where $|P|$ is less that the least positive
critical value of $|P|$. Then, a general
lower bound on the size and distance between the connected components
of the zero locus of $P$
(corollary \ref{sphereinside} and theorem \ref{distancebetween}).
 
The last section will present experimental results among which are five extremal sextic curves far
from the discriminant. These are obtained by very long running
numerical optimisation (many months) some of which are not terminated.

\tableofcontents

\pagebreak

\parindent=0pt
\parskip=5pt

\section{Notation}

Let $\mathcal{S}^{n-1}$ be the unit sphere of
$\mathbb{R}^n$. We write
$\|x\|$ the usual Euclidean norm on $\mathbb{R}^n$.

We consider $\mathbb{E} = \mathbb{R}[X_1,\dots,X_n]_d$ the vector space of homogeneous
polynomials in $n > 1$ variables of degree
$d > 1$. Let $N$ be the dimension of this vector space, we have $N = \binom{d+n-1}{n-1}
\geq n$

Let $\langle\_,\_\rangle$ be a scalar product on $\mathbb{E}$ and
$\|\_\|$ the associated norm. 
We use the same notation for the scalar product and norm of
$\mathbb{E}$ as for $\mathbb{R}^n$, the context should make it clear what norm
we are using. 

Let $\mathcal{B} = (E_1,\dots,E_N)$ be an orthonormal basis of 
$\mathbb{E}$.  

For $x \in \mathbb{R}^n$, $C(x)$ denotes the line 
vector 
$(E_1(x),\dots,E_N(x))$ and  $B_i(x)$ for $i \in \{1,\dots,n\}$ denotes the line 
vector 
$(\frac{\partial E_1(x)}{\partial x_i},\dots,\frac{\partial 
  E_N(x)}{\partial x_i})$. Let $B(x)$ be the $n \times N$ matrix 
whose lines are $B_i(x)$ for $i \in \{1,\dots,n\}$. 

For $P \in\mathbb{E}$, let $P_{\mathcal{B}}$ be the column vector 
coordinates of $P$ in the basis $\mathcal{B}$. We may write: 
$$ 
P(x) = {C(x)P_{\mathcal B}}, \frac{\partial 
  P(x)}{\partial x_i} = {B_i(x)P_{\mathcal B}} \hbox{ and } 
\nabla P(x) = B(x)P_{\mathcal B} 
$$ 

We will also use the following notation for the normal and tangent
component of a vector field $V(x)$ defined for $x \in
\mathcal{S}^{n-1}$:
\begin{align*}
V^N(x) &= \langle x | V(x)\rangle x \\ 
V^T(x) &= V(x) - V(x)^N
\end{align*}
In the particular case of $\nabla P(x)$, we
write $\nabla^T P(x)$ and we have Euler's relation $\nabla^N P(x) = d P(x) x$, which
gives:
$$
\nabla P(x) = \nabla^T P(x) + d P(x) x \hbox{ with }  \langle
\nabla^T P(x) | x \rangle = 0
$$

Similarly, we write $\mathcal H P(x)$ for the hessian matrix of $P$ at
$x$. We have that 
\begin{align*}
{}^tV \mathcal H P(x) x =  {}^tx \mathcal H P(x) V &=
  (d-1) \langle \nabla P(x) | V \rangle \\
&=  (d-1) \langle \nabla^T P(x) | V\rangle + d(d-1)P(x)\langle x | V\rangle \\
\text{and } {}^tx \mathcal H P(x) x &= d (d-1) P(x)\|x\|^2 \cr
\end{align*}

Hence, we can find a symmetrix matrix $\mathcal H^T P(x)$ whose kernel
contains $x$ and such that~:
\begin{align*}
 {}^tV \mathcal H P(x) V
= d (d-1) P(x) \langle x | V \rangle^2 + 2 (d-1) \langle \nabla^T
P(x) | V \rangle \langle x | V \rangle +  {}^tV \mathcal H^T P(x) V
\end{align*}

Geometrically, $\mathcal H^T P(x)$ is the matrix of the linear
application defined as $\pi(x) \circ \nabla^2 P(x) \circ \pi(x)$ where
$x \mapsto \pi(x)$ is the projection on the plane tangent to the unit sphere at
$x$ and $\nabla^2 P(x)$ is the second derivative of $P$ seen as a linear
application.

\begin{fact} 
The matrix $B(x)$ is always of maximal rank (i.e. of rank 
$n$) for all $x \neq 0$.  
\end{fact} 

\begin{proof} 
Let us prove first that $B(x)$ is of maximal rank when the elements of $\mathcal 
B$ are monomials with arbitrary coefficients. By symmetry, we may assume that $x_1 
\neq 0$. Thus, $B(x)$ contains the 
following columns 
coming from the partial derivatives of 
$a_i x_1^{n-1}x_i$ for $1 \leq i \leq n$: 
$$ 
\left( 
\begin{matrix} 
         a_1 n x_1^{n-1}& a_2 (n-1) x_1^{n-2}x_2 &  a_3 (n-1) 
         x_1^{n-2}x_3&\dots& a_n (n-1) x_1^{n-2}x_n\cr 
         0 &a_2 x_1^{n-1}& 0&\dots& 0\cr 
         0 &0&a_3 x_1^{n-1}& \dots&0\cr 
         \vdots &\vdots &\vdots& \ddots&  \vdots\cr 
         0 & 0 & 0 & \dots & a_n x_1^{n-1}\cr 
\end{matrix} 
\right) 
$$ 
This proves that the  $n$ lines of $B(x)$ are linearly 
independent when the basis contains only monomials. Second, 
If for some basis $B(x)$ where of rank less that $n$, this  
would yield a linear combination with some non zero coefficients 
such that $ \sum_{1\leq i \leq n} \lambda_i B_i(x) = 0$, implying that 
for any polynomial $P$ we would have  $(\lambda_1,\dots,\lambda_n) 
B(x) P_{\mathcal B} = \sum_{1\leq i \leq n} \lambda_i 
\frac{\partial P}{\partial x_i}(x) = 0$, and this being independent 
of the basis would mean that $B(x)$ is never of maximal rank for that $x$.  
\end{proof} 


\section{Distance to the real discriminant}

\begin{definition}
The real discriminant $\Delta$ of the space $\mathbb{E}$ of polynomials of degree $d$ in $n$ variables is the set
of polynomials $P \in \mathbb{E}$ such that there exists
$x \in \mathcal{S}^{n-1}$ where
$P(x)=0$ and $\nabla P(x)=0$.
\end{definition}

This can be written 
$$
\Delta = \bigcup_{x \in \mathcal{S}^{n-1}} \Delta_x \hbox{ where }
\Delta_x = \{ P \in \mathbb{E}; B(x) P_{\mathcal B} = 0 \hbox{ and } C(x) P_{\mathcal B}  = 0 \}
$$

As usual, the equation $C(x) P_{\mathcal B} = 0$ is redundant because
of the Euler's relation which can be written here $C(x) = \frac{1}{d} (x_1,\dots,x_n) B(x)$. 

Therefore, the discriminant $\Delta$ is a union of sub-vector spaces of $\mathbb{E}$ of codimension
$n$ (given that $B(x)$ is of maximal rank).

Let  $P$ be a given polynomial in $\mathbb{E}$. We give a way to compute the 
distance between $P$ and $\Delta$.

We first choose $x_0 \neq 0$ and we compute
the distance from $P$ to $\Delta_{x_0}$. Therefore, we look for $Q \in \mathbb{E}$, such that:
\begin{itemize}
\item $P + Q \in \Delta_{x_0}$.
\item $\|Q\|$ minimal.
\end{itemize}

The first condition may be written 
$$
B(x_0)(P_{\mathcal B} + Q_{\mathcal B}) = 0
$$

The second condition is equivalent to  $Q$ orthogonal to
$\Delta_{x_0}$, which means that $Q_{\mathcal B}$ is a linear combination of the
vectors $ ^tB_i(x_0)$, the columns of $ ^tB(x_0)$. 

This means that there exists a column vector $H$ of size $n$
such that 
$$Q_{\mathcal B} =  {^t\!B(x_0)}H.$$
This gives:
$$
B(x_0)P_{\mathcal B} +  {B(x_0)}{ ^t\!B(x_0)}H = 0
$$

Let us define 
$$
A(x) = B(x){^t\! B(x)} \hbox{ and } M(x) = A(x)^{-1} 
$$
$B(x)$ is a $n \times N$ matrix of maximal rank with $n \leq N$. This implies
that $A(x)$ is an $n \times n$ symmetrical and definite matrix for all $x \neq
0$. Hence, $M(x)$ is well defined and symmetrical.

We have
$$
B(x_0)P_{\mathcal B} +  {A(x_0)}H = 0 \hbox{ which
  implies }
H = {-}M(x_0)B(x_0)P_{\mathcal B}
$$
and
$$
Q_{\mathcal B} = {-} {^t\!B(x_0)}M(x_0)B(x_0)P_{\mathcal B}
$$

We can now write the distance to $\Delta_{x_0}$ by
 
\begin{align*}
\dist^2(P,\Delta_{x_0}) &= \|Q\|^2 \\
  &= {^t\!Q_{\mathcal B}} Q_{\mathcal B}  \\
  &= {^t\! P_{\mathcal B}}  {^t\!B(x_0)}M(x_0)B(x_0) {^t\!B(x_0)}M(x_0)B(x_0) P_{\mathcal B}\\
  &= {^t\! P_{\mathcal B}}  {^t\!B(x_0)}M(x_0)A(x_0)M(x_0)B(x_0) P_{\mathcal B}\\
  &= {^t\! P_{\mathcal B}}  {^t\!B(x_0)}M(x_0)B(x_0) P_{\mathcal B}\\
  &= {^t \nabla P(x_0)} M(x_0) \nabla P(x_0)
\end{align*}

The above formula, established for any $x_0 \neq 0$, is homogeneous in
$x_0$. We can
therefore state our first lemma:
\begin{lemma}\label{distgen}
Let $(E_1,\dots,E_N)$ be an orthornomal basis of $\mathbb{E} =
\mathbb{R}[X_1,\dots,X_n]_d$ for a given scalar
product. Let $B(x)$ be the $n \times N$ matrix defined by:
 $$
B(x) = \left(\frac{\partial E_j(x)}{\partial
    x_i}\right)_{\begin{array}{l}\scriptstyle 1 \leq i \leq
    n\cr\scriptstyle 1 \leq j \leq N\cr\end{array}} 
$$
For any homogeneous polynomial $P \in \mathbb E$, 
the distance to the discriminant $\Delta$ associated to the given scalar
product is given by
$$
\dist(P,\Delta) = \min_{x \in  \mathcal{S}^{n-1}} \sqrt{{^t\nabla
    P(x)} M(x) \nabla P(x)} \hbox{ with }
M(x) = (B(x){^tB(x)})^{-1}
$$
\end{lemma}


\section{The Bombieri norm}

The above lemma can be simplified in the particular case of Bombieri
norm\cite{BBEM90}. To do so, we recall the definition and properties of Bombieri norm and
scalar product.

Notation: let $\alpha = (\alpha_i,\dots,\alpha_n)$ be a vector in
$\mathbb{N}^n$ and $x = (x_1,\dots,x_n)\in
\mathbb{R}^n$, we write:
\begin{itemize}
\item $|\alpha| = \Sigma_{i=1}^n \alpha_i = d$,
\item $ \alpha! = \Pi_{i=1}^n \alpha_i!$,
\item $x^\alpha =  \Pi_{i=1}^n x_i^{\alpha_i}$ for $x \in \mathbb{R}^n$,
\item $\chi_i = (0,\dots,0,1,0,\dots,0)$ where the index of $1$ is $i$.
\end{itemize}

\begin{definition}[Bombieri norm and scalar product]
The Bombieri scalar product \cite{BBEM90} for homogeneous polynomial of degree $d$
is defined by 
$$
\|x^\alpha\|^2 = \frac{\alpha!}{|\alpha|!} \hbox{ and } \langle x^\alpha|x^\beta \rangle = 0 \hbox{ if } \alpha  \neq \beta
$$
\end{definition}

The Bombieri scalar product and the associated norm have the
remarkable property to be invariant by the action of the orthogonal
group of $\mathbb{R}^n$. It was originally introduced because it
verifies the Bombieri inequalities for product of
polynomials. However, we do not use this property here.
 
We now give a
lemma establishing the invariance and a result we need later in this
article:
\begin{lemma}\label{scalarlinear}
Let $\{u_i\}_{1 \leq i \leq d}$ and $\{v_i\}_{1 \leq i \leq d}$ be two
families of vectors of ${\mathbb R}^n$. 
Let us consider the two following homogeneous polynomials in $\mathbb E$:
\begin{align*}
U(x) &= \prod_{1 \leq i \leq d} \langle x | u_i\rangle &
V(x) &= \prod_{1 \leq i \leq d} \langle x | v_i\rangle \\
\end{align*}
The Bombieri scalar product of these polynomials is given by the
following formula which directly relates the Bombieri scalar product
of polynomials to the Euclidian one in $\mathbb R^n$:
$$
\langle U | V \rangle = \frac{1}{d!} \sum_{\sigma \in S_n} \prod_{1 \leq i \leq d}
 \langle u_i | v_{\sigma(i)} \rangle 
$$
When the two families are constant i.e. $U(x) =\langle x |
u\rangle^d$ and  $V(x) =\langle x |v\rangle^d$, this simplifies to:
$$
\langle U | V \rangle =  \langle u | v \rangle^d
$$
\end{lemma}

\begin{proof}
We start by developing the polynomials $U$ and $V$. 
For this, we use $\rho, \rho'$ to denote applications from
$\{1,\dots,d\}$ to  $\{1,\dots,n\}$ and we write
$M(\rho) \in \mathbb N^n$ the vector such that
$M_i(\rho) = \mathrm{Card}(\rho^{-1}(\{i\}))$.

\begin{align*}
\langle U | V \rangle
 &= \left\langle \sum_{|\alpha|=d} x^\alpha \sum_{M(\rho) = \alpha}
  \prod_{1 \leq i \leq d} u_{i,\rho(i)}  \middle|
 \sum_{|\alpha|=d} x^\alpha \sum_{M(\rho') = \alpha}
  \prod_{1 \leq j \leq d} v_{j,\rho'(j)} \right\rangle \\
 &=  \sum_{|\alpha|=d} \frac{\alpha!}{|\alpha|!} 
\sum_{M(\rho) = M(\rho') = \alpha} \prod_{1 \leq i,j \leq d}
u_{i,\rho(i)} v_{j,\rho'(j)} \\
 &= \sum_{|\alpha|=d} \frac{1}{|\alpha|!} 
  \sum_{\sigma \in S_d} \sum_{M(\rho) = \alpha} \prod_{1 \leq i \leq d}
u_{i,\rho(i)} v_{\sigma{i},\rho(i)} \\
 &\hspace{10em} \hbox{ Using the $\alpha!$ permutations in $S_d$
  such that $\rho' = \rho \circ \sigma$ }\\
 &= \frac{1}{d!} \sum_{\sigma \in S_d} \sum_{|\alpha|=d} \sum_{M(\rho)
  = \alpha} \prod_{1 \leq i \leq d}
u_{i,\rho(i)} v_{\sigma{i},\rho(i)} \\
 &= \frac{1}{d!} \sum_{\sigma \in S_d} \prod_{1 \leq i \leq d} \langle
u_i | v_{\sigma(i)} \rangle \\
\end{align*}
\end{proof}

\begin{corollary}
The Bombieri norm is invariant by composition with the 
orthogonal group.
\end{corollary}

\begin{proof}
Proving this corollary is just proving that the Bombieri norm does not
depend upon the choice of coordinates in $\mathbb R^n$. The last
theorem establishes this for product of linear forms that generate all polynomials.
\end{proof}

We also have the following corollary, which is a way to see the
Veronese embedding in the particular case of Bombieri norm:
\begin{corollary}\label{veronese}
Let $P$ be an homogeneous polynomial of degree $d$ with $n$ variables,
then we have
$$
P(u) = \langle P | U \rangle \hbox{ with } U(x) = \langle x | u \rangle^d 
$$
\end{corollary}

\begin{proof}
If we write $P$ has a linear combination of
monomials, the lemma \ref{scalarlinear}
immediately gives the result.
\end{proof}

We will use the following inequality which are proved in appendix \ref{appbomb}:
\begin{lemma}\label{bombineq}
For all $P \in \mathbb{E}$ and all $x \in \mathbb{R}^n$, we have:
$$
\begin{array}{rcl}
|P(x)| &\leq& \|P\| \|x\|^{d} \cr
\|\nabla P(x)\| &\leq& d \, \|P\| \,\|x\|^{d-1} \cr
\|\mathcal{H} P(x)\|_2 \leq \|\mathcal{H} P(x)\|_F &\leq&
\displaystyle d(d-1) \,\|P\|\, \|x\|^{d-2} \cr
\end{array}
$$
Using the following norms:
\begin{itemize}
\item The Euclidian norm on $\mathbb R^n$ (for $x$ and $\nabla P(x)$),
\item The Bombieri norm for polynomials (for $P$)
\item The Frobenius norm written $\|\_\|_F$ which is the square root
  of the sum of the squares of the matrix coefficients (for the
  Hessian $\mathcal{H} P(x)$).
\item The spectral norm written $\|\_\|_2$ which is the largest absolute value
  of the eigenvalues of the matrix (also for the
  Hessian $\mathcal{H} P(x)$).
\end{itemize}
\end{lemma}

All this inequalities are equalities for the monomial $x_i^d$ for $1
\leq i \leq n$ and by invariance for $d$ power of linear form.
In this case, the Hessian matrix will have only one non
null eigenvalue which implies that
$\|\mathcal{H} P(x)\|_2 = \|\mathcal{H} P(x)\|_F$.


\section{Distance with Bombieri norm}

Here is the formulation of the lemma \ref{distgen} in the particular
case of Bombieri's norm. It can be established from lemma
\ref{distgen}, but we propose a more direct proof using the invariance
by composition with the orthogonal group.

\begin{theorem}\label{bombdist}
Let $P \in \mathbb E$ be an homogeneous polynomial of degree $d$ with
$n$ variables. The distance to the real
discriminant $\Delta$ for the Bombieri norm is given by:
$$
\dist(P,\Delta) = \min_{x \in {\mathcal S}^{n-1}}\sqrt{ \displaystyle P(x)^2 +
  \frac{\|\nabla^T P(x)\|^2}{d} }
$$
\end{theorem}

\begin{proof}
Consider $c \in \mathcal S^{n-1}$. We want to compute
$\dist(P,\Delta_c)$.
One can always find $h$ an element of the orthogonal group such that
$$
h(0,\dots,0,1) = c \hbox{ and } 
h(1,0,\dots,0) = \frac{\nabla^T P(c)}{\|\nabla^T P(c)\|} \hbox{ which
  implies } 
$$
\begin{align}
{P \circ h}(x) &= P(c) x_n^d + \|\nabla^T P(c)\| x_1 x_n^{d-1} + Q(x) \label{bombdisteq} 
\end{align}
where the monomials $x_n^d$ and $x_i x_n^{d-1}$ for $i \in
\{1,\dots,n\}$ do not appear in $Q(x)$.

Then, using the fact that the Bombieri norm is invariant by isometry,
the fact that distinct monomials are othogonal and the fact that $Q
\in \Delta_{(0,\dots,0,1)}$ which implies that $Q \circ h^{-1} \in \Delta_c$, we
have:
\begin{align}
\dist^2(P,\Delta_c) &= \dist^2(P \circ h, Q)\nonumber\\
     &=\|P(c) x_n^d + \nabla^T P(c) x_1 x_n^{d-1}\|^2\nonumber\\
&=
 \displaystyle P(c)^2 +
  \frac{\|\nabla^T P(c)\|^2}{d}  \label{bombintereqone}
\end{align}

We can also give an alternate formulation avoiding the decomposition
of the gradient in normal and tangent components:

\begin{align}
\dist^2(P,\Delta_c)
  &= \displaystyle P(c)^2 + \frac{\|\nabla^T P(c)\|^2}{d} \nonumber\\
  &= \displaystyle \frac{\|\nabla^N P(c)\|^2}{d^2} +
  \frac{\|\nabla^T P(c)\|^2}{d}  \nonumber \\
  &= \displaystyle  \frac{\|\nabla^N P(c)\|^2}{d^2} - \frac{\|\nabla^N P(c)\|^2}{d} +
  \frac{\|\nabla P(c)\|^2}{d}  \nonumber\\ &= \label{bombintereqdeux}
 \displaystyle (1-d) P(c)^2 +
  \frac{\|\nabla P(c)\|^2}{d}
\end{align}
\end{proof}

Let us define from equation (\ref{bombintereqdeux}) $\delta_P(x) = \frac{\|\nabla P(x)\|^2}{d} - (d-1)P(x)^2$. In
the theorem \ref{bombdist}, it is enough to consider
the critical points of $\delta_P$ on the unit sphere, that is points where $\nabla^T \delta_P(x)
= 0$. This means we have:

$$
\dist(P,\Delta) = \min_{x \in \mathcal S^{n-1}, \nabla^T \delta_P(x)
  = 0} \sqrt{\delta_P(x)} 
$$

Using $\mathcal H P(x) x = (d-1)\nabla P(x)$ and $\langle \nabla P(x)
| x \rangle = d P(x)$, we compute:
\begin{align}
\frac{d}{2} \nabla \delta_P(x) &= \mathcal H P(x) \nabla P(x) - d (d-1) P(x)
\nabla P(x)  \nonumber \\
 &=  \mathcal H P(x) \nabla P(x) - \langle \nabla P(x) | x \rangle
\mathcal H P(x) x \nonumber \\
 &=  \mathcal H P(x) (\nabla P(x) - \langle \nabla P(x) | x \rangle
 x) \nonumber \\
 &= \mathcal H P(x) \nabla^T P(x) \nonumber \\
 &= \mathcal H^T P(x) \nabla^T P(x) + (d-1) \|\nabla^T P(x)\|^2 x \label{hessianeq}
\end{align}

The first term in (\ref{hessianeq}) is $\frac{d}{2} \nabla^T \delta_P(x)$. Hence, we have:

\begin{align}
\dist(P,\Delta) &=& \min_{x \in \mathcal S^{n-1},  \mathcal H^T P(x) \nabla^T P(x)
  = 0} \sqrt{P(x)^2 +
  \frac{\|\nabla^T P(x)\|^2}{d}} \label{hesseq}
\end{align}

This motivates the following definition:

\begin{definition}[quasi-singular points, contact polynomial, contact radius]\label{maindef}
We will call quasi-singular points for $P \in \mathbb E$ the critical points of 
$\delta_d$ with norm 1 where the distance to the discriminant is reached. 
This means that
\begin{center}
$c \in \mathcal S^{n-1}$ is a quasi-singular points iff
  $\dist(P,\Delta) = \delta_P(c)$.
\end{center}
A necessary condition for $c$ to be a quasi singular point of $P$ is
$$\mathcal H^T P(c) \nabla^T P(c)  = 0$$

We will say that $Q$ is a {\em contact polynomial} for $P$ at $c$ if $c$ is a
quasi-singular point for $P$, $Q \in \Delta_c$ (this means that $\{x
\in \mathcal S^{n-1}; Q(x) = 0\}$ has a
singularity at $c$) and $\dist(P,\Delta) = \|Q - P\|$.

When $Q$ is contact polynomial for $P$ at $c$, we will say that $R =
Q - P$ is a {\em contact radius} for $P$ at $c$. A contact radius $R$ is therefore the
smallest polynomial for Bombieri norm that must be added to $P$ to
create a singularity.

Then, we distinguish two  kinds of quasi-singular points for $P$ (their names
will be explaned later):
\begin{description}
\item[quasi-double points] $c$ is quasi-double point if it is a
  quasi-singular point of $P$ and a critical point of $P$  on the
  unit sphere (i.e. satisfying 
$\nabla^T P(c) = 0$).
\item[quasi-cusp points] $c$ is quasi-cups point for $P$ if it is a
  quasi-singular point of $P$ which is not a
  critical point of $P$. In this case, $\nabla^T P(c)$ is a non zero 
member of the kernel  of $\mathcal H^T P(c)$.
\end{description}
\end{definition}

First, using the quasi-double points, we can find a very simple inequality
for the distance to the discriminant:

\begin{theorem}\label{bombineqdist}
Let $P \in \mathbb E$ be an homogeneous polynomial of degree $d$ with
$n$ variables. The distance to the real
discriminant $\Delta$ for the Bombieri norm satisfies:
$$
\dist(P,\Delta) \leq \min_{x \in {\mathcal S}^{n-1}, \nabla^T P(x) = 0}|P(x)|
$$

The condition $\nabla^T P(x) = 0$ means that $x$ is a critical
point of $P$ and our theorem means that the distance to the
discriminant is less or equal to the minimal critical value of $P$ in absolute
value.
\end{theorem}

\begin{proof}
 We use 
the theorem \ref{bombdist}:

\begin{align}
\dist^2(P,\Delta) &= \displaystyle \min_{x \in {\mathcal
    S}^{n-1}} \left(P(x)^2 + \frac{\|\nabla^T P(x)\|^2}{d}\right) \cr
&\leq \displaystyle \min_{x \in {\mathcal
    S}^{n-1}, \nabla^T P(x) = 0}P(x)^2 
\end{align}

\end{proof}

\begin{theorem}\label{contactradius}
Let $P \in \mathbb E$ be an homogeneous polynomial of degree $d \geq 2$ with
$n$ variables. Let $c$ be a quasi-singular point for $P$. Then, the contact
radius at $c$ is the polynomial 
$$R(x) = - P(c) \langle x \mid c \rangle^d - 
         {\langle x \mid \nabla^T P(c) \rangle} \langle x \mid c \rangle^{d-1} .$$

and $Q(x) = P(x) + R(x)$, the contact polynomial for $P$ at $c$, has no
other singularity than $c$ and $-c$.

Moreover, when $d = 2$, $c$ is always a quasi double point (i.e. $\nabla^T P(c) = 0$).
\end{theorem}

\begin{proof}
The formula for $R(x)$
is a consequence of the equation \ref{bombdisteq} established in the
proof of theorem \ref{bombdist} (given just after the
theorem).

Let us assume that $Q$ has another singularity $c' \neq c$ and $c'
\neq -c$ on the unit sphere (recall that we imposed quasi-singular
point to lie on the unit sphere). This
means that $\dist(P,\Delta_c) = \dist(P,\Delta_{c'})$, $Q$ lying at
the intersection of $\Delta_c$ and $\Delta_{c'}$.

We can therefore write $Q(x) = P(x) + S(x)$, where $S$ is the contact radius at $c'$:
$$S(x) = - P(c') \langle x \mid c' \rangle^d - \langle x \mid \nabla^T
P(c') \rangle \langle x \mid c' \rangle^{d-1} .$$

We necessarily have $S = R$. It remains to show that this is impossible. We have:
\begin{align*}
R(x) &= - \langle x \mid c \rangle^{d-1} \langle x \mid P(c) c +
\nabla^T P(c) \rangle \\
S(x) &= - \langle x \mid c' \rangle^{d-1} \langle x \mid P(c') c' +
\nabla^T P(c') \rangle \\
\end{align*}

When $d \geq 3$, 
the hyper-surface $R(x) = 0$ contains the plane   $\langle x \mid c
\rangle = 0$ with multiplicity $d-1$ union the plane 
 $\langle x \mid P(c) c +  \nabla^T P(c)\rangle = 0$ with multiplicity
one. $S(x) = 0$ uses that same plane with $c$ replaced by $c'$,  which imposes $c = c'$ or $c = -c'$.

When $d = 2$, we will show in the study of quasi-cusp point that they
exist only from degree $3$, hence we know that we only have
quasi-double points, which means that  $\nabla^T P(c) = \nabla^T P(c')
= 0$. Therefore, $R$
and $S$ become:
\begin{align*}
R(x) &= - P(c) \langle x \mid c \rangle^{d} \\
S(x) &= - P(c')\langle x \mid c' \rangle^{d} \\
\end{align*}

And again, $R = S$ implies $c=c'$ or $c =-c'$.
\end{proof}

\subsection{Study of quasi-double points}

Let $P \in \mathbb E$ be an homogeneous polynomial of degree $d$ with
$n$ variables. Let $c$ be a quasi-double point for $P$, meaning that
we have $\nabla^T P(c) = 0$ and $\dist^2(P,\Delta) = P(c)^2 > 0$.

The Bombieri norm being invariant
by the orthogonal group, using a rotation we can assume that $c =
(0,\dots,0,1)$ and that the matrix $\mathcal H^T P(c)$ is diagonal.

Knowing that $\nabla^T P(c) = 0$, we can write:
$$
P(x) = \alpha x_n^d + \frac{1}{2} \sum_{1 \leq i < n} \lambda_1 x_i^2 x_n^{d-2} + T(x) \hbox{ with } \alpha = P(c)\hbox{ and }
\lambda_i =  \frac{\partial^2
  P}{\partial x_i^2}(c)
$$
with no monomial of degree $\leq 2$ in $x_1,...,x_{n-1}$ in $T(x)$,
i.e. $T$ has valuation at least $3$ in $x_1,...,x_{n-1}$.

Then, by theorem \ref{contactradius}, the contact radius is 
$$R(x) = - \alpha \langle x \mid c \rangle^d$$

and the contact polynomial is

$$Q(x) = P(x) + R(x) = \frac{1}{2} \sum_{1 \leq i < n}  \lambda_i x_i^2 x_n^{d-2} + T(x)
$$

The singularity at $c$ of the variety $\{ x \in {\mathcal S}^{n-1} | Q(x)
= 0\}$ is at least a double point (justifying the name quasi-double
point) and it has no other singularities by theorem~\ref{contactradius}.

Next, we will reveal some constraints on the eigenvalues $\lambda_i = \frac{\partial^2
  P}{\partial x_i^2}(c)$ of the hessian matrix. For this, we consider the point $$c_h =
\frac{1}{\sqrt{1+h^2}}(h,0,\dots,0,1)$$ and compute 
$(\dist^2(P,\Delta_{c_h}) - \dist^2(P,Q))(1+h^2)^d$ which is non negative
 because $\dist(P,\Delta_{c_h}) \geq \dist(P,Q)$.

\begin{align*}
\left(\dist^2(P\right.&,\Delta_{c_h}) -
\left.\dist^2(P,Q)\right)(1+h^2)^d \\ 
=& \;  \left((1-d) P^2(c_h) + \frac{\|\nabla P(c_h)\|^2}{d} - P(c)^2\right)(1+h^2)^d \\
=& \; (1-d) P^2(h,0,\dots,0,1) \\
 & + \frac{\|\nabla P(h,0,\dots,0,1)\|^2}{d}(1+h^2) - P(c)^2(1+h^2)^d \\
=& \; (1-d)(\alpha + \frac{1}{2}\lambda_1h^2 + o(\|h\|^2))^2\\
 & + \frac{\left(d \alpha + (d-2)
    \frac{1}{2}\lambda_1 h^2 + o(\|h\|^2)\right)^2 + \left(\lambda_1 h +
    o(\|h\|)\right)^2}{d}(1+h^2) \\
 & - \alpha^2(1+d h^2+o(\|h\|^2)) \\
=& \; ((1-d) + d - 1)\alpha^2 +  ((1 - d) + (d - 2))\alpha \lambda_1
  h^2 +  \frac{1}{d} \lambda_1^2 h^2 \\
& + d\alpha^2h^2 - d\alpha^2h^2 + o(\|h\|^2) \\
=& \; \left(- \alpha \lambda_1 + \frac{1}{d} \lambda_1^2\right) h^2 + o(\|h\|^2)
\end{align*}

Therefore, $\dist(P,\Delta_{c_h}) > \dist(P,\Delta)$ implies:

$$ \lambda_1 \left(\lambda_1  -  d\alpha\right)  \geq 0$$

The same is true for all the eigenvalues and this means that
when $\lambda_i$ and $P(c)$ have the same sign then 
$|\lambda_i| \geq d |P(c)|$ (recall that by definition $\alpha = P(c)$).




This study establishes the following theorem:

\begin{theorem}\label{quasidouble}
Let $P \in \mathbb E$ be an homogeneous polynomial of degree $d$ with
$n$ variables, let $c$ be a quasi-double point for $P$ and $Q$ a corresponding contact
polynomial at $c$. Then, the contact radius is 
$$
R(x) = - P(c)\langle x | c\rangle^d
$$
The contact polynomial $Q(x) = P(x) + R(x)$ has only one singularity in $c$
on ${\mathcal S}^{n-1}$ which is at
least a double-point. 

Moreover, if $\lambda$ is an eigenvalue of $\mathcal H^T P(x)$
with the same sign than $P(c)$, then $|\lambda| \geq d|P(c)| > 0$.
\end{theorem}

\subsection{Study of quasi-cusp point}

Let $P \in \mathbb E$ be an homogeneous polynomial of degree $d$ with
$n$ variables. Let $c$ be a quasi-cusp point for $P$, meaning that
we have $\nabla^T P(c) \neq 0$ and $\mathcal H^T P(c) \nabla^T P(c) =
0$.

The Bombieri norm being invariant
by the orthogonal group, using a rotation we can assume that $c =
(0,\dots,0,1)$ and that the matrix $\mathcal H^T P(c)$ is diagonal and
that $(0,1,0,\dots,0)$ is the direction of $\nabla^T P(c)$ which is an
eigenvector of $\mathcal H^T P(c)$.

We can write:
$$
P(x) = \alpha x_n^d + \beta x_1x_n^{d-1} + \frac{1}{2}\sum_{2 \leq i < n} \lambda_i x_i^2 x_n^{d-2} +  \frac{1}{6}\mu_1 x_1^3 x_n^{d-3} + T(x)
$$
with $\alpha = P(c)$, $\beta = \frac{\partial
  P}{\partial x_1}(c)$, $\lambda_i =    \frac{\partial^2
  P}{\partial x_i^2}(c)$, $\mu_1 =  \frac{\partial^3
  P}{\partial x_1^3}(c)$ and no monomial of degree $\leq 2$ in $x_1,...,x_{n-1}$, nor $x_1^3 x_n^{d-3}$ in $T(x)$.

The fact that the coefficient of $x_1^2 x_n^{d-1}$ is null
is the condition $\mathcal H^T P(c) \nabla^T P(c) =
0$.
 
Then, by theorem \ref{contactradius}, the contact radius is 
$$R(x) = - \alpha x_n^d - \beta x_1x_n^{d-1}  $$

and the contact polynomial is

$$Q(x) = P(x) + R(x) = \sum_{2 \leq i < n   } \frac{1}{2} \lambda_i x_i^2
x_n^{d-2} +  \frac{1}{6}\mu_1 x_1^3 x_n^{d-3} + T(x)
$$

The singularity at $c$ of $Q(x) = 0$ on the
unit sphere is at least a cusp (justifying the name quasi-cusp
point) and it has no other singularities by theorem~\ref{contactradius}.

We now use a computation similar to the previous case to reveal
a constraint on $\mu_1$.
For this, we consider the point $c_h =
\frac{1}{\sqrt{1+h^2}}(h,0,\dots,0,1)$ and compute 
$(\dist^2(P,\Delta_{c_h}) - \dist^2(P,Q))(1+h^2)^d$ which is non negative because $\dist(P,\Delta_{c_h}) \geq \dist(P,Q)$.

\begin{align*}
\left(\dist^2(P\right. &,\Delta_{c_h})- \left.\dist^2(P,Q)\right)(1+h^2)^d\\ &=
\left((1-d) P^2(c_h) + \frac{\|\nabla P(c_h)\|^2}{d} - P(c)^2 -  \frac{\|\nabla^T P(c)\|^2}{d}\right)(1+h^2)^d \\
&= 
(1-d) P^2(h,0,\dots,0,1) + \frac{\|\nabla
    P(h,0,\dots,0,1)\|^2}{d}(1+h^2) \\ &- \left(\alpha^2 +  \frac{\beta^2}{d}\right)(1+h^2)^d \\
&= (1-d)(\alpha + \beta h + o(\|h\|^2))^2\\ &+ \frac{(d \alpha + (d-1)
    \beta h + o(\|h\|^2))^2 + (\beta + \frac{1}{2} \mu_1 h^2 + o(\|h\|^2))^2}{d}(1+h^2)
  \\ &- \left(\alpha^2 +  \frac{\beta^2}{d}\right)(1+d h^2 +o(\|h\|^2)) \\
&= ((1-d) + d - 1)\alpha^2 +  \left(\frac{1}{d} - \frac{1}{d}\right)\beta^2 +
 (2(1 - d) + 2(d - 1))\alpha \beta
  h \\ &+  \left((1 - d) + \frac{(d-1)^2}{d} + \frac{1}{d} - 1\right) \beta^2 h^2 +
  \frac{1}{d} \beta \mu_1 h^2 + o(\|h\|^2) \\
&= \left(\frac{2 - 2d}{d} \beta^2 + \frac{1}{d} \beta\mu_1\right) h^2 + o(\|h\|^2)
\end{align*}

Therefore, $\dist(P,\Delta_{c_h}) > \dist(P,\Delta)$ implies:

$$\beta (2 (1-d) \beta + \mu_1) \geq 0$$

This forces $\beta \mu_1 > 0$ hence $\mu_1 \neq 0$ (because $d =
2$).
 
We remark that if $d = 2$, then $\mu_1 = 0$ and 
together with $\beta \neq 0$, this implies $\dist(P,\Delta_{c_h}) <
\dist(P,\Delta)$ for $h$ small enough.
This proves that quasi-cusp points exist only when $d > 2$. This computation does not requires theorem~\ref{contactradius} (we just use
the fact that $c$ is a local minima of $\delta_P$). This
fills the gap in the proof of theorem~\ref{contactradius} for the degree $2$.

It remains to explicit the constraints on the eigenvalues $\lambda_i = \frac{\partial^2
  P}{\partial x_i^2}(c)$ for $2 \leq i < n$. They change compared to the
case of quasi-double points. In this case, we have to take into
account the coefficient of $x_1 x_i^2 x_n^{d-3}$ for $2 \leq i < n$ which is $\frac{1}{2} \mu_i$ with $\mu_i = \frac{\partial^3
  P}{\partial x_1 \partial x_i^2}(c)$. 

For this, we consider the point $c_h =
\frac{1}{\sqrt{1+h^2}}(0,h,0,\dots,0,1)$ and compute 
$(\dist^2(P,\Delta_{c_h}) - \dist^2(P,Q))(1+h^2)^d$ which is non negative
because $\dist(P,\Delta_{c_h}) \geq \dist(P,Q)$.

\begin{align*}
\left(\dist(P\right. &,\Delta_{c_h}) - \left.\dist(P,Q)\right)(1+h^2)^d\\ &=
\left((1-d) P^2(c_h) + \frac{\|\nabla P(c_h)\|^2}{d} - P(c)^2 -  \frac{\|\nabla^T P(c)\|^2}{d}\right)(1+h^2)^d \\
&=
(1-d) P^2(0,h,0,\dots,0,1) + \frac{\|\nabla
    P(0,h,0,\dots,0,1)\|^2}{d}(1+h^2) \\ &- \left(\alpha^2 +  \frac{\beta^2}{d}\right)(1+h^2)^d \\
&= (1-d)(\alpha + \frac{1}{2} \lambda_2 h^2 + o(\|h\|^2))^2
\\ & + \frac{(d
  \alpha + (d-2) \frac{1}{2}
    \lambda_2 h^2 + o(\|h\|^2))^2}{d}(1+h^2)
\\ & + \frac{
   (\beta + \frac{1}{2} \mu_2 h^2 + o(\|h\|^2))^2 + (\lambda_2 h + o(\|h\|^2))^2}{d}(1+h^2)
  \\ &- \left(\alpha^2 +  \frac{\beta^2}{d}\right)(1+dh^2+o(\|h\|^2)) \\
&= ((1-d) + d - 1)\alpha^2 +  \left(\frac{1}{d} - \frac{1}{d}\right)\beta^2 +
 ((1 - d) + (d - 2))\alpha \lambda_2
  h^2 \\ &+ \left(\frac{1}{d} - 1\right)\beta^2 h^2 + \frac{1}{d}\lambda_2^2 h^2 +  \frac{1}{d}\beta\mu_2 h^2 + o(\|h\|^2) \\
&= \frac{1}{d}((1 - d) \beta^2 - d\alpha\lambda_2+  \lambda_2^2+ \beta\mu_2) h^2 + o(\|h\|^2)
\end{align*}

Therefore, $\dist(P,\Delta_{c_h}) > \dist(P,\Delta)$ implies:

$$(1 - d) \beta^2 - d\alpha\lambda_2 +  \lambda_2^2+ \beta\mu_2  \geq 0$$

Hence, if $\lambda_2 = 0$ we have $\mu_2 \neq 0$ with the same sign as
$\beta$.

By symmetry, the same holds for $\lambda_i$ with $i \geq 2$.
Moreover, up to reordering, we may assume that $\lambda_2 = \dots =
\lambda_k = 0$ and that $\lambda_i \neq 0$ for $k < i < n$. In fact, $k+1$ is
the dimension of the kernel of the matrix ${\mathcal H}^T P$, this is
  at least 2, because in contains at least $(1,0,\dots,0)$ and
  $(0,\dots,0,1)$.

Then, we
consider the hessian matrix of $\frac{\partial P}{\partial x_1}$,
restricted to the variables $x_{1},\dots,x_{k}$ and consider a 
change of coordinates such that this matrix is diagonal. In such a
coordinates system, we can write:

$$
P(x) = \alpha x_n^d + \beta x_1x_n^{d-1} + \frac{1}{2} \sum_{k < i < n}
\lambda_i x_i^2 x_n^{d-2} + \frac{1}{6} \mu_1 x_1^3 x_n^{d-3} +
\frac{1}{2} \sum_{2 \leq i \leq k} \mu_i x_1 x_i^2 x_n^{d-3} + T(x)
$$
where $T$ has no monomial of degree less than $3$ in 
$x_1,\dots,x_{n-1}$ and no monomial of degree $3$ in
$x_1,\dots,x_{n-1}$, using only the variables $x_1,\dots,x_{k}$.

This study allows us to state the following theorem:

\begin{theorem}\label{cuspcase}
Let $P \in \mathbb E$ be an homogeneous polynomial of degree $d$ with
$n$ variables. Let $c$ be a quasi-cusp point for $P$. Then, the contact
radius at $c$ is the polynomial 
$$R(x) = - P(c) \langle x \mid c \rangle^d - \langle x \mid \nabla^T
P(c) \rangle \langle x \mid c \rangle^{d-1} .$$

The contact polynomial $Q(x) = P(x) + R(x)$ has only one singularity in $c$ which is at
least a cusp.

We also have
$$
\mathcal H^T P(c).\nabla^T P(c) = 0 \hbox{ and } \nabla^T P(c) \neq 0
$$

Moreover, we can choose coordinates where 
$c = (0,\dots,0,1)$, $\nabla^T P(c) = (\beta,0,\dots,0)$ and $k + 1
\geq 2$ is
the dimension of the kernel of the matrix ${\mathcal H}^T P(c)$ ($c$
and $\nabla^T P(c)$ are in the kernel of ${\mathcal H}^T P(c)$) and
$$
P(x) = \alpha x_n^d + \beta x_1x_n^{d-1} + \frac{1}{2} \sum_{k < i < n}
\lambda_i x_i^2 x_n^{d-2} + \frac{1}{6} \mu_1 x_1^3 x_n^{d-3} +
\frac{1}{2} \sum_{2 \leq i \leq k} \mu_i x_1 x_i^2 x_n^{d-3} + T(x)
$$
where $T$ has no monomial of degree less than $3$ in 
$x_1,\dots,x_{n-1}$ and no monomial of degree $3$ in
$x_1,\dots,x_{n-1}$, using only the variables $x_1,\dots,x_{k}$.
We also have the following constraints:
\begin{itemize}
\item $\beta, \mu_1, \dots, \mu_k$ are non zero and have the same
  sign and
\item $(1 - d) \beta^2 -d\alpha\lambda_i +  \lambda_i^2+ \beta\mu_i
  \geq 0$ for $2 \leq i < n$ where $\mu_i = \frac{\partial^3
    P}{\partial x_1 \partial^2 x_i}(c)$.
\item $\lambda_i \neq 0$ for $k < i < n$.
\end{itemize}
\end{theorem}


\section{Application to extremal hyper-surfaces}

\begin{definition}[Extremal and maximal hyper-surfaces]
An hyper-surface on the projective space or the unit sphere of dimension $n-1$ defined by an equation $P(x) =
0$ where $P$ is an homogeneous polynomial of degree $d$ in $n$ variables 
is \emph{extremal} if the tuple of its Betti numbers $(b_0, \dots, b_{n-2})$
is maximal for pointwise ordering for such polynomials. 

We say that such an hyper-surface is \emph{maximal} when the sum of
its Betti numbers is maximal.
\end{definition}

Remark: considering the same polynomial on the projective space or the
sphere just doubles the Betti numbers.

The next theorem also applies to locally extremal surface:
\begin{definition}[Locally extremal hyper-surfaces]
An algebraic hyper-surface $\mathcal H$ in the projective plane or the unit sphere
of dimension $n-1$ is \emph{locally
extremal} if there exists no algebraic hyper-surface of the same degree
isotopic to $\mathcal H$ with 
a disc $D^{n-1}$ (whose border is $S^{n-2}$) replaced by another
surface with the same border and greater Betti numbers than the disc. This
definition includes the addition of new connected components.
\end{definition}

It is clear that an extremal hyper-surface is locally extremal (because
doing a connected sum mostly corresponds to adding Betti numbers). But
the converse is not true in general. For instance the plane sextic
curve with nine ovals where 2 or 6 lie in another oval are locally extremal
but not extremal, nor maximal.

\begin{theorem}\label{extremal}
Let $P \in \mathbb E$ be an homogeneous polynomial of degree $d$ in
$n$ variables. Assume that the zero level of $P$ on the unit sphere is
smooth and locally extremal. Then, we have:
\begin{itemize}
\item $P$ admit no quasi-cusp point.
\item If $c$ is a quasi-double point of $P$, then at least one of the
  eigenvalue $\lambda$ of $\mathcal H^T P(c)$ for an eigen vector
  distinct from $c$ itself satisfies $\lambda P(c)
  \leq 0$ (we always have $(\mathcal H^T P(c)) c = 0$ by definition of $\mathcal H^T$).
\end{itemize}
\end{theorem}

\begin{proof}
Let $P \in \mathbb E$ be an homogeneous polynomial of degree $d$ in
$n$ variables with a smooth and
locally extremal zero locus on the unit sphere.

For the first item, assume that $c$ is a quasi-cusp of $P$, and that
$R$ and $Q$ are respectively the contact radius and polynomial of $P$
at $c$.

By theorem \ref{cuspcase}, we can find
coordinates where $c = (0,\dots,0,1)$ and
$$
P(x) = \alpha x_n^d + \beta x_1x_n^{d-1} + \frac{1}{2} \sum_{k < i < n}
\lambda_i x_i^2 x_n^{d-2} + \frac{1}{6} \mu_1 x_1^3 x_n^{d-3} +
\frac{1}{2} \sum_{2 \leq i \leq k} \mu_i x_1 x_i^2 x_n^{d-3} + T(x)
$$
with the properties given above by theorem \ref{cuspcase}
and especially $\beta \mu_i > 0$ for $1 \leq i \leq k$.

First, without loss of generality, we can assume $\alpha \geq 0$
(by considering $-P$ instead of $P$) and $\beta > 0$ (using the
transformation $x_1 \mapsto -x_1$). We furthermore reorder variables
and define $m \in \NN$ to have
\begin{itemize}
\item $\lambda_{k+1},\dots,\lambda_m > 0$
\item $\lambda_{m+1},\dots,\lambda_{n-1} < 0$.
\end{itemize}

We will study and change the topology of the zero level of $P$ in a
neighbourhood of the point $c:(0,\dots,0,1)$. Hence we will work till the end of the
proof with affine coordinates and set $x_n = 1$.

We will study the following families of polynomials (only the
coefficient of $x_1$ is changing):
\begin{align*}
P_t(x) = &\; \alpha t^5 + \beta t x_1 +  \sum_{2 \leq i \leq k}
t x_i^2 + \frac{1}{2} \sum_{k < i < n}
\lambda_i x_i^2 \\
& + \frac{1}{6} \mu_1 x_1^3 +
\frac{1}{2} \sum_{2 \leq i \leq k} \mu_i x_1 x_i^2  + T(x) \\
P_t^+(x) = &\; \alpha t^5 + \beta t^3 x_1 +  \sum_{2 \leq i \leq k}
t x_i^2 + \frac{1}{2} \sum_{k < i < n}
\lambda_i x_i^2 \\
& + \frac{1}{6} \mu_1 x_1^3 +
\frac{1}{2} \sum_{2 \leq i \leq k} \mu_i x_1 x_i^2  + T(x) \\
P_t^-(x) = &\; \alpha t^5 - \beta t^3 x_1 +  \sum_{2 \leq i \leq k}
t x_i^2 + \frac{1}{2} \sum_{k < i < n}
\lambda_i x_i^2 \\
& + \frac{1}{6} \mu_1 x_1^3 +
\frac{1}{2} \sum_{2 \leq i \leq k} \mu_i x_1 x_i^2  + T(x)
\end{align*}

We have $\dist^2(P,P_t) = 
\alpha^2 (1 - t^{5})^2 + \frac{\beta^2}{d} (1 - t)^2 +
\frac{2(k-1)}{d(d-1)} t^2 =  \alpha^2 + \frac{\beta^2}{d}(1 - 2t) + o(t)$ and this is smaller that $\dist^2(P,\Delta)
= \alpha^2 + \frac{\beta^2}{d}$ for $t$ small enough. This implies
that $P_t(x) = 0$ has the same topology than $P(x) = 0$ for $t \in ]0,
  \epsilon[$
for some $\epsilon > 0$ (1).

For $t$ small enough, the topologies of $P_t(x) = 0$, $P_t^+(x) = 0$
and $P_t^-(x) = 0$
 can be computed
using Viro's theorem \cite{Viro83,Viro84}. To do so, we attribute a
height to each vertex of Newton's polyhedra: if
$x_1^{\alpha_1},\dots,x_{n-1}^{\alpha_{n-1}}$ is a monomial of $P_t$,
we consider the point $(\alpha_1,\dots,\alpha_{n-1},h_\alpha) \in
\NN^n$. To simplify the discussion, we will identify the 
point  $(\alpha_1,\dots,\alpha_{n-1},h_\alpha) \in
\NN^n$ with the corresponding monomial.

All monomials are given $0$ height except
$1$ which we place at height $5$, $x_i^2$ for $2 \leq i \leq k$ which we place
at height $1$, $x_1$ which height changes among the
three families.

The triangulation needed by Viro's theorem is computed as the projection
of the convex hull of the points of Newton's
polytopes with their given height. It is easy to see that all vertices
with non zero coefficient are on the convex hull, just looking at the
axes.

In what follows, we consider that $t$ is small enough  to
have (1) and for the
topologies of $P_t(x) = 0$, $P_t^+(x) = 0$
and $P_t^-(x) = 0$ to be given by Viro's theorem, gluing the
topologies of the polynomial in each polyhedron.

Hence, we only need to consider polyhedra changing among the three
polynomials. The only vertices that belong to a polyhedra which is not the same for
the Viro's decomposition of $P_t$, $P^+_t$ and $P^-_t$ are among
\begin{itemize}
\item $1$, $x_1$, $x_1^2$, $x_1^3$,
\item $x_1 x_i^2$ for $2 \leq i \leq k$,
\item $x_i^2$ for $k < i < n$ and
\item monomials without $x_1$.
\end{itemize}

This is true because the only monomial that changes height is
$x_1$. Therefore, a changing polyhedron, that contains a monomial
$x^\alpha$ must contain a segment from $x^\alpha$ to $x_1$.
Because the vertices of $x_1^3$, $x_1 x_i^2$ for $2 \leq i \leq k$ and
$x_i^2$ for $k < i < n$ are at height $0$, if $x^\alpha$ is not among
those, it must satisfies $\alpha_1 = 0$ because otherwise the
segment joining $x^\alpha$ to $x_1$ can not have 0 height
when it crosses the simplex corresponding to
 $x_1^3$, $x_1 x_i^2$ for $2 \leq i \leq k$ and
$x_i^2$ for $k < i < n$.

This means that to study the change of topology of $P_t(x) = 0$, $P_t^+(x) = 0$
and $P_t^-(x) = 0$, we can consider that $T$ is constant in $x_1$. 

The rest of the proof is in two steps: we already know that $P(x) = 0$
and  $P_t(x) = 0$ have the same topology. It remains to show that
 $P_t(x) = 0$, $P_t^+(x) = 0$ also have the same topology and that 
$P_t^-(x) = 0$, compared to $P_t^-(x) = 0$, has at least two Betti numbers that increase while the
others are non decreasing.

For the first two polynomials: $P(x) = 0$
and  $P_t(x) = 0$, they have the same topology because, considering
that $T$ is constant in $x_1$, they can be written~:
\begin{align*}
 P_t(x) = \frac{1}{6} \mu x_1^3 + \left(\beta t + \frac{1}{2} 
  \sum_{1 < i \leq k} \!\!\mu_i x_i^2\right) x_1 + \left(\alpha t^5 + \frac{1}{2} 
  \sum_{2 \leq i \leq k} \!\!t x_i^2 + \frac{1}{2} 
  \sum_{k < i < n} \!\!\lambda_i x_i^2 + T(x)\right)
\end{align*}

and 

\begin{align*}
 P_t^+(x) = \frac{1}{6} \mu x_1^3 + \left(\beta t^3 + \frac{1}{2} 
  \sum_{1 < i \leq k} \!\!\mu_i x_i^2\right) x_1 + \left(\alpha t^5 + \frac{1}{2} 
  \sum_{2 \leq i \leq k} \!\!t x_i^2 + \frac{1}{2} 
  \sum_{k < i < n} \!\!\lambda_i x_i^2 + T(x)\right)
\end{align*}

In both cases, all non constant coefficients in $x_1$ are positive,
implying that the polynomial has exactly one root because its
discriminant in $x_1$ is negative. This means that $P_t(x) = 0$ and 
$P_t^+(x) = 0$ defines a graph of $x_1$ as a function of the other
variables and hence have the same topology. Moreover, the infinite
branch are the same in both cases, not changing the gluing with 
the polyhedra corresponding to neglected monomials in $T$ (those using $x_1$).

Finally, for the change of topology between $P_t^+(x) = 0$ and
$P_t^-(x) = 0$, we only need to consider the following polyhedra which
are simplices:
\begin{description}
\item[$A$] with vertices $1$, $x_1$ and $x_i^2$ for $1 < i < n$.
\item[$B$] with vertices $x_1$, $x_1^3$ and $x_i^2$ for $1 < i < n$.
\end{description}

It is easy to check that the chosen height for the monomials forces
these simplices to appear.

In the case of $P_t^+(x) = 0$, all coefficients are positive (see
figure \ref{viroone}), which leads to a disc of dimension $n-2$ inside
the polyhedra $A$ and $B$, regardless of the sign of the
coefficient $\lambda_i$, again because the discriminant is negative.

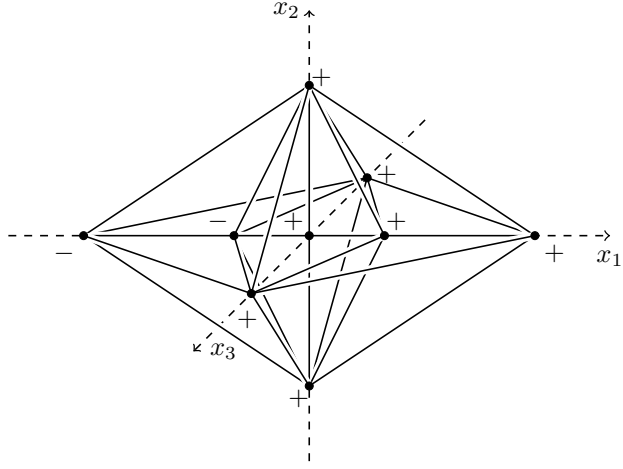
\begin{figure}
\begin{tikzpicture} 

\draw[->,dashed,semithick] (-4,0,0) -- (4,0,0); 
\draw[->,dashed,semithick] (0,-3,0) -- (0,3,0); 
\draw[->,dashed,semithick] (0,0,-4) -- (0,0,4); 

\draw (4,-0.3,0) node {$x_1$}; 
\draw (-0.3,3,0) node {$x_2$}; 
\draw (0.4,0,4) node {$x_3$}; 

\draw[semithick] (0,0,-2) -- (3,0,0); 
\draw[semithick] (0,0,-2) -- (-3,0,0); 
\draw[semithick] (0,0,-2) -- (0,2,0); 
\draw[semithick] (0,0,-2) -- (0,-2,0); 

\draw[semithick] (0,0,-2) -- (1,0,0); 
\draw[semithick] (0,0,-2) -- (-1,0,0);


\draw[line width=3pt,white] (-3,0,0) -- (3,0,0);
\draw[line width=3pt,white] (0,-2,0) -- (0,2,0);

\draw[line width=3pt,white] (3,0,0) -- (0,2,0);
\draw[line width=3pt,white] (0,2,0) -- (-3,0,0);
\draw[line width=3pt,white] (3,0,0) -- (0,-2,0);
\draw[line width=3pt,white] (0,-2,0) -- (-3,0,0);

\draw[line width=3pt,white] (1,0,0) -- (0,2,0);
\draw[line width=3pt,white] (0,2,0) -- (-1,0,0);
\draw[line width=3pt,white] (1,0,0) -- (0,-2,0);
\draw[line width=3pt,white] (0,-2,0) -- (-1,0,0);

\draw[semithick] (-3,0,0) -- (3,0,0);
\draw[semithick] (0,-2,0) -- (0,2,0);
\draw[semithick] (3,0,0) -- (0,2,0);
\draw[semithick] (0,2,0) -- (-3,0,0);
\draw[semithick] (3,0,0) -- (0,-2,0);
\draw[semithick] (0,-2,0) -- (-3,0,0);
\draw[semithick] (1,0,0) -- (0,2,0);
\draw[semithick] (0,2,0) -- (-1,0,0);
\draw[semithick] (1,0,0) -- (0,-2,0);
\draw[semithick] (0,-2,0) -- (-1,0,0);


\draw[line width=3pt,white] (0,0,2) -- (1,0,0); 
\draw[line width=3pt,white] (0,0,2) -- (-1,0,0); 
\draw[line width=3pt,white] (0,0,2) -- (0,-2,0); 
\draw[line width=3pt,white] (0,0,2) -- (0,2,0); 
\draw[line width=3pt,white] (0,0,2) -- (3,0,0); 
\draw[line width=3pt,white] (0,0,2) -- (-3,0,0); 
\draw[semithick] (0,0,2) -- (-1,0,0); 
\draw[semithick] (0,0,2) -- (1,0,0); 
\draw[semithick] (0,0,2) -- (0,2,0); 
\draw[semithick] (0,0,2) -- (0,-2,0); 
\draw[semithick] (0,0,2) -- (3,0,0); 
\draw[semithick] (0,0,2) -- (-3,0,0);

\draw[fill=black] (0,0,0) circle (0.15em) 
	    node[above left,xshift=0.05cm,yshift=-0.05cm] {$+$}; 
\draw[fill=black] (1,0,0) circle (0.15em) 
	    node[above right,xshift=-0.1cm,yshift=-0.05cm] {$+$}; 
\draw[fill=black] (3,0,0) circle (0.15em) 
	    node[below right] {$+$}; 

\draw[fill=black] (0,0,2) circle (0.15em) 
	    node[below,xshift=-0.05cm,yshift=-0.1cm] {$+$}; 
\draw[fill=black] (0,0,-2) circle (0.15em) 
	    node[right,yshift=0.05cm] {$+$}; 

\draw[fill=black] (0,2,0) circle (0.15em) 
	    node[above right,xshift=-0.1cm,yshift=-0.13cm] {$+$}; 
\draw[fill=black] (0,-2,0) circle (0.15em) 
	    node[below left,xshift=0.12cm,yshift=0.07cm] {$+$}; 

\draw[fill=black] (-1,0,0) circle (0.15em) 
	    node[above left,xshift=0.05cm,yshift=-0.05cm] {$-$}; 
\draw[fill=black] (-3,0,0) circle (0.15em) 
	    node[below left] {$-$};

\end{tikzpicture} 
\caption{Topology of $P_t^+(x) = 0$ near $c$, with $k = n$.}\label{viroone}
\end{figure}

In the case of $P_t^-(x) = 0$, only the sign
of $x_1^2$ changes. The change is illustrated by figure \ref{virotwo}
and \ref{virothree}.

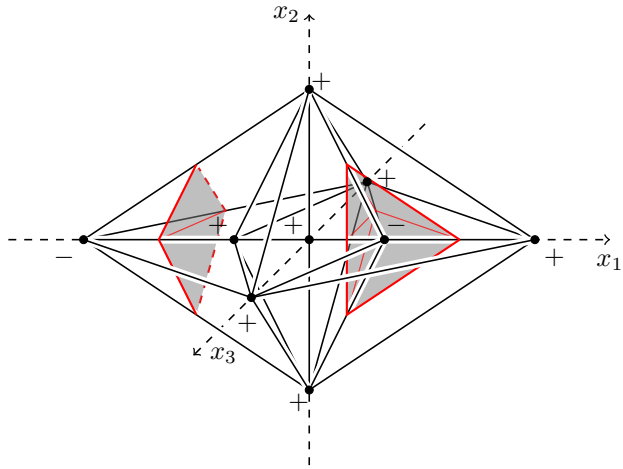
\begin{figure}
\begin{tikzpicture} 

\draw[->,dashed,semithick] (-4,0,0) -- (4,0,0); 
\draw[->,dashed,semithick] (0,-3,0) -- (0,3,0); 
\draw[->,dashed,semithick] (0,0,-4) -- (0,0,4); 

\draw (4,-0.3,0) node {$x_1$}; 
\draw (-0.3,3,0) node {$x_2$}; 
\draw (0.4,0,4) node {$x_3$}; 

\draw[semithick] (0,0,-2) -- (3,0,0); 
\draw[semithick] (0,0,-2) -- (-3,0,0); 
\draw[semithick] (0,0,-2) -- (0,2,0); 
\draw[semithick] (0,0,-2) -- (0,-2,0); 

\draw[semithick] (0,0,-2) -- (1,0,0); 
\draw[semithick] (0,0,-2) -- (-1,0,0); 

\draw[red,thick,dashed] (-1.5,1,0) -- (-1.5,0,-1) --(-1.5,-1,0);

\fill[gray,opacity=0.5] (2,0,0) -- (0.5,0,-1) -- (0.5,1,0);
\fill[gray,opacity=0.5] (0.5,0,0) -- (0.5,0,-1) -- (0.5,1,0);
\fill[gray,opacity=0.5] (2,0,0) -- (0.5,0,-1) -- (0.5,-1,0);
\fill[gray,opacity=0.5] (0.5,0,0) -- (0.5,0,-1) -- (0.5,-1,0);
\fill[gray,opacity=0.5] (-1.5,0,-1) -- (-2,0,0) -- (-1.5,1,0);
\fill[gray,opacity=0.5] (-1.5,0,-1) -- (-2,0,0) -- (-1.5,-1,0);

\draw[red,thin,opacity=0.5] (2,0,0) -- (0.5,0,-1) -- (0.5,1,0);
\draw[red,thin,opacity=0.5] (0.5,0,0) -- (0.5,0,-1) -- (0.5,1,0);
\draw[red,thin,opacity=0.5] (2,0,0) -- (0.5,0,-1) -- (0.5,-1,0);
\draw[red,thin,opacity=0.5] (0.5,0,0) -- (0.5,0,-1) -- (0.5,-1,0);
\draw[red,thin,opacity=0.5] (-1.5,0,-1) -- (-2,0,0) -- (-1.5,1,0);
\draw[red,thin,opacity=0.5] (-1.5,0,-1) -- (-2,0,0) -- (-1.5,-1,0);

\draw[line width=3pt,white] (-3,0,0) -- (3,0,0);
\draw[line width=3pt,white] (0,-2,0) -- (0,2,0);

\draw[line width=3pt,white] (3,0,0) -- (0,2,0);
\draw[line width=3pt,white] (0,2,0) -- (-3,0,0);
\draw[line width=3pt,white] (3,0,0) -- (0,-2,0);
\draw[line width=3pt,white] (0,-2,0) -- (-3,0,0);

\draw[line width=3pt,white] (1,0,0) -- (0,2,0);
\draw[line width=3pt,white] (0,2,0) -- (-1,0,0);
\draw[line width=3pt,white] (1,0,0) -- (0,-2,0);
\draw[line width=3pt,white] (0,-2,0) -- (-1,0,0);

\draw[semithick] (-3,0,0) -- (3,0,0);
\draw[semithick] (0,-2,0) -- (0,2,0);
\draw[semithick] (3,0,0) -- (0,2,0);
\draw[semithick] (0,2,0) -- (-3,0,0);
\draw[semithick] (3,0,0) -- (0,-2,0);
\draw[semithick] (0,-2,0) -- (-3,0,0);
\draw[semithick] (1,0,0) -- (0,2,0);
\draw[semithick] (0,2,0) -- (-1,0,0);
\draw[semithick] (1,0,0) -- (0,-2,0);
\draw[semithick] (0,-2,0) -- (-1,0,0);

\draw[red,thick] (2,0,0) -- (0.5,1,0) -- (0.5,0,0)  -- (0.5,-1,0) --
(2,0,0);
\draw[red,thick] (-1.5,1,0) -- (-2,0,0) -- (-1.5,-1,0);

\draw[line width=3pt,white] (0,0,2) -- (1,0,0); 
\draw[line width=3pt,white] (0,0,2) -- (-1,0,0); 
\draw[line width=3pt,white] (0,0,2) -- (0,-2,0); 
\draw[line width=3pt,white] (0,0,2) -- (0,2,0); 
\draw[line width=3pt,white] (0,0,2) -- (3,0,0); 
\draw[line width=3pt,white] (0,0,2) -- (-3,0,0); 
\draw[semithick] (0,0,2) -- (-1,0,0); 
\draw[semithick] (0,0,2) -- (1,0,0); 
\draw[semithick] (0,0,2) -- (0,2,0); 
\draw[semithick] (0,0,2) -- (0,-2,0); 
\draw[semithick] (0,0,2) -- (3,0,0); 
\draw[semithick] (0,0,2) -- (-3,0,0);

\draw[fill=black] (0,0,0) circle (0.15em) 
	    node[above left,xshift=0.05cm,yshift=-0.05cm] {$+$}; 
\draw[fill=black] (1,0,0) circle (0.15em) 
	    node[above right,xshift=-0.1cm,yshift=-0.05cm] {$-$}; 
\draw[fill=black] (3,0,0) circle (0.15em) 
	    node[below right] {$+$}; 

\draw[fill=black] (0,0,2) circle (0.15em) 
	    node[below,xshift=-0.05cm,yshift=-0.1cm] {$+$}; 
\draw[fill=black] (0,0,-2) circle (0.15em) 
	    node[right,yshift=0.05cm] {$+$}; 

\draw[fill=black] (0,2,0) circle (0.15em) 
	    node[above right,xshift=-0.1cm,yshift=-0.13cm] {$+$}; 
\draw[fill=black] (0,-2,0) circle (0.15em) 
	    node[below left,xshift=0.12cm,yshift=0.07cm] {$+$}; 

\draw[fill=black] (-1,0,0) circle (0.15em) 
	    node[above left,xshift=0.05cm,yshift=-0.05cm] {$+$}; 
\draw[fill=black] (-3,0,0) circle (0.15em) 
	    node[below left] {$-$}; 

\end{tikzpicture} 
\caption{Topology of $P_t^-(x) = 0$ near $c$  with $\lambda_3 > 0$. Only the
  rear faces are shown, rear border dashed.}\label{virotwo}
\end{figure}

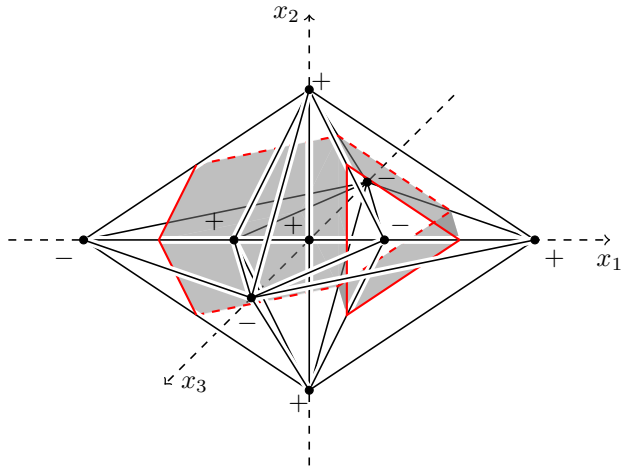
\begin{figure}
\begin{tikzpicture} 

\draw[->,dashed,semithick] (-4,0,0) -- (4,0,0); 
\draw[->,dashed,semithick] (0,-3,0) -- (0,3,0); 
\draw[->,dashed,semithick] (0,0,-5) -- (0,0,5); 

\draw (4,-0.3,0) node {$x_1$}; 
\draw (-0.3,3,0) node {$x_2$}; 
\draw (0.4,0,5) node {$x_3$}; 

\draw[semithick] (0,0,-2) -- (3,0,0); 
\draw[semithick] (0,0,-2) -- (-3,0,0); 
\draw[semithick] (0,0,-2) -- (0,2,0); 
\draw[semithick] (0,0,-2) -- (0,-2,0); 

\draw[semithick] (0,0,-2) -- (1,0,0); 
\draw[semithick] (0,0,-2) -- (-1,0,0);

\fill[gray,opacity=0.5] (2,0,0) -- (1.5,0,-1) --(0,1,-1) -- (0.5,1,0); 
\fill[gray,opacity=0.5] (0.5,0,0) -- (0,0,-1) -- (0,1,-1) --(0.5,1,0); 

\fill[gray,opacity=0.5] (2,0,0) -- (1.5,0,-1) --(0,-1,-1) -- (0.5,-1,0); 
\fill[gray,opacity=0.5] (0.5,0,0) -- (0,0,-1) -- (0,-1,-1) --(0.5,-1,0); 

\fill[gray,opacity=0.5] (-1.5,1,0) -- (0,1,-1) -- (-1.5,1,0);
\fill[gray,opacity=0.5] (0,1,-1) -- (-0.5,0,-1)-- (0,0,-1); 
\fill[gray,opacity=0.5] (-2,0,0)--(-1.5,1,0) -- (0,1,-1) -- (-0.5,0,-1); 

\fill[gray,opacity=0.5] (-1.5,-1,0) -- (0,-1,-1) -- (-1.5,-1,0);
\fill[gray,opacity=0.5] (0,-1,-1) -- (-0.5,0,-1)-- (0,0,-1); 
\fill[gray,opacity=0.5] (-2,0,0)--(-1.5,-1,0) -- (0,-1,-1) -- (-0.5,0,-1); 



\draw[red,thick,dashed] (-1.5,1,0) -- (0,1,-1) -- (1.5,0,-1);
\draw[red,thick,dashed] (-1.5,-1,0) -- (0,-1,-1) -- (1.5,0,-1);



\draw[line width=3pt,white] (-3,0,0) -- (3,0,0);
\draw[line width=3pt,white] (0,-2,0) -- (0,2,0);

\draw[line width=3pt,white] (3,0,0) -- (0,2,0);
\draw[line width=3pt,white] (0,2,0) -- (-3,0,0);
\draw[line width=3pt,white] (3,0,0) -- (0,-2,0);
\draw[line width=3pt,white] (0,-2,0) -- (-3,0,0);

\draw[line width=3pt,white] (1,0,0) -- (0,2,0);
\draw[line width=3pt,white] (0,2,0) -- (-1,0,0);
\draw[line width=3pt,white] (1,0,0) -- (0,-2,0);
\draw[line width=3pt,white] (0,-2,0) -- (-1,0,0);

\draw[semithick] (-3,0,0) -- (3,0,0);
\draw[semithick] (0,-2,0) -- (0,2,0);

\draw[semithick] (3,0,0) -- (0,2,0);
\draw[semithick] (0,2,0) -- (-3,0,0);
\draw[semithick] (3,0,0) -- (0,-2,0);
\draw[semithick] (0,-2,0) -- (-3,0,0);

\draw[semithick] (1,0,0) -- (0,2,0);
\draw[semithick] (0,2,0) -- (-1,0,0);
\draw[semithick] (1,0,0) -- (0,-2,0);
\draw[semithick] (0,-2,0) -- (-1,0,0);

\draw[red,thick] (-1.5,1,0) -- (-1.5,1,0) --(-2,0,0) -- (-1.5,-1,0) --(-1.5,-1,0);
\draw[red,thick] (2,0,0) -- (0.5,1,0) -- (0.5,0,0) -- (0.5,-1,0) -- (2,0,0)     ;

\draw[line width=3pt,white] (0,0,2) -- (1,0,0); 
\draw[line width=3pt,white] (0,0,2) -- (-1,0,0); 
\draw[line width=3pt,white] (0,0,2) -- (0,-2,0); 
\draw[line width=3pt,white] (0,0,2) -- (0,2,0); 
\draw[line width=3pt,white] (0,0,2) -- (3,0,0); 
\draw[line width=3pt,white] (0,0,2) -- (-3,0,0); 
\draw[semithick] (0,0,2) -- (-1,0,0); 
\draw[semithick] (0,0,2) -- (1,0,0); 
\draw[semithick] (0,0,2) -- (0,2,0); 
\draw[semithick] (0,0,2) -- (0,-2,0); 
\draw[semithick] (0,0,2) -- (3,0,0); 
\draw[semithick] (0,0,2) -- (-3,0,0); 

\draw[fill=black] (0,0,0) circle (0.15em) 
	    node[above left,xshift=0.05cm,yshift=-0.05cm] {$+$}; 
\draw[fill=black] (1,0,0) circle (0.15em) 
	    node[above right,xshift=-0.05cm,yshift=-0.05cm] {$-$}; 
\draw[fill=black] (3,0,0) circle (0.15em) 
	    node[below right] {$+$}; 

\draw[fill=black] (0,0,2) circle (0.15em) 
	    node[below,xshift=-0.05cm,yshift=-0.1cm] {$-$}; 
\draw[fill=black] (0,0,-2) circle (0.15em) 
	    node[right,yshift=0.05cm] {$-$}; 

\draw[fill=black] (0,2,0) circle (0.15em) 
	    node[above right,xshift=-0.1cm,yshift=-0.13cm] {$+$}; 
\draw[fill=black] (0,-2,0) circle (0.15em) 
	    node[below left,xshift=0.12cm,yshift=0.07cm] {$+$}; 

\draw[fill=black] (-1,0,0) circle (0.15em) 
	    node[above left] {$+$}; 
\draw[fill=black] (-3,0,0) circle (0.15em) 
	    node[below left] {$-$}; 
\end{tikzpicture} 
\caption{Topology of $P_t^-(x)=0$ near $c$  with $k=2$ or $\lambda_2 > 0$ and $\lambda_3 < 0$. Only the
  rear faces are shown, rear border dashed.}\label{virothree}
\end{figure}

We show that the topology of the hyper-surface $P_t^-(x) = 0$ in the
polyhedra $A$ and $B$ and their counterparts in all orthants is a disc
with a handle.
The two polyhedra $A$ and $B$ being simplices, the topology is given by
the sign at the vertices.

First, we see that the polynomial admits three roots 
$\phi^- < 0 \leq  \phi^0 < \phi^+$ on the $x_1$ axes.

In the dimension $x_1, \dots, x_m$, the monomial $x_1$ is negative surrounded
by positive monomials in polyhedra $A$ and $B$. This gives us a component $S_A$ homeomorphic to a sphere  of dimension
$m-1$, 
inside the hypersurface $P_t^-(x) = 0$ and containing $\phi^0$ and $\phi^+$. Moreover, we can also
find a topological sphere $S'_A$ (by inflating $S_A$ a little) that does not
meet the hyper-surface $P_t^-(x) = 0$ and that
contains a point on the $x_1$ axes between $\phi^-$ and $\phi^0$.

Similarly, In the dimension $x_1, x_{m+1}, \dots, x_{n-1}$, the 
monomials $x_1$ and $1$ ($-x_1$ alone if $\alpha = 0$)
are positive and surrounded by negative vertices.

This gives us a component $S_B$ homeomorphic to a sphere  of dimension
$n-m-1$,
inside the hypersurface $P_t^-(x) = 0$ and containing $\phi^-$ and $\phi^0$. Moreover, we can also
find a topological sphere $S'_B$ (by inflating $S_B$ a little) that does not meet the same
hyper-surface and that
contains a point on the $x_1$ axes between $\phi^0$ and $\phi^+$.

Now, the sum of the dimensions of the spheres $S_A$ and $S_B$ is $m - 1 + n-m-1 = n-2$
which is one less than the dimension of the ambient space $\mathbb
R^{n-1}$. This means we can compute the linking number of $S_A$ and
$S'_B$ (resp. $S_B$ and $S'_A$). It may be computed as the
intersection of $S'_B$ and $D_A$, the disc inside $S_A$ in the $m$
first dimensions. 

This intersection number is $1$ because the intersection is $\{\phi^0\}$  and this indicates that 
the spheres $S_A$ (resp. $S_B$) is not homotope to $0$ (i.e. non
contractile) in $S^{n-1}
\setminus S'_B$ (resp. $S^{n-1}
\setminus S'_A)$ hence not homotope to $0$ in the zero locus of $P_t^-(0)$. Therefore,
the presence of $S_A$ and $S_B$ ensures that we have at least an
hyper-surface, inside the polyhedra $A$ and $B$, with the two Betti numbers $b_{m-1}$ and $b_{n-m-1}$
which are positive (if $m = 1$ or $m = n-1$, $b_0 > 1$). This can not be just a disc.

This means that $P_t^+(x) = 0$ is a desingularisation of $Q(x)= 0$ that
creates a disc while $P_t^-(x) = 0$ creates a disc with at least one
handle.
But, $P_t^+(x) = 0$ gives us the topology of $P(x)=0$.
This establishes the equation $P(x) = 0$ does not define a locally
extremal hyper-surface.
 
The last part of the theorem is easier: a quasi-double point $c$ for $P$ 
such that all eigen values of $\mathcal H^T P(c)$ (except $c$ itself)
satisfies $\lambda P(c) > 0$ would mean that $c$ is a local minimum of
$|P(c)|$ and therefore, $c$ is an isolated point
 of the hyper-surface $Q(x) = 0$ where $Q \in \Delta$ is the contact
 polynomial $Q(x) = P(x) - P(c)\langle x | c\rangle^d$ for $P$. This allows to add a
 new connected component to the variety of equation $P(x) = 0$. This
 is also impossible in the case of a locally extremal algebraic hyper-surface.
\end{proof}

\begin{corollary}\label{bombeqdistbis}
Let $P \in \mathbb E$ be an homogeneous polynomial of degree $d$ with
$n$ variables. Assume that the zero level of $P$ is
locally extremal. Then, the distance to the discriminant is the
minimal absolute critical value of $P$ i.e.

$$
\dist(P,\Delta) = \min_{\nabla^T P(x) = 0} |P(x)|
$$

Remark: this implies that the right member of the above equality is
continuous in the coefficient of $P$ which is not true in general.
\end{corollary}

\begin{proof}
Immediate from the first item of the previous theorem, the
definition \ref{maindef} and the equation \ref{hesseq} that precedes it.
\end{proof}


\section{Further from the discriminant}\label{further}

We now establish a property verified by polynomials that
maximise the distance to the discriminant: 
\begin{theorem}\label{comblin}
Let $P$ be an homogenous polynomial in $n$ variables, of
degree $d$, Bombieri norm $1$ and such that
$\dist(P,\Delta)$ is locally maximal among polynomials of Bombieri
norm $1$. 

Let  $\{c_1,-c_1,\dots,c_k,-c_k\}$ 
be the set of the quasi-singular points of $P$.
Let $R_1,\dots,R_k$ be the corresponding contact radius (we choose one
for each pair $(c,-c)$ because the contact radius corresponding to $c$ and $-c$
are equal or opposite, depending upon the degree).

Then, $P$ is a linear combination of the $R_i$.
\end{theorem}

\begin{proof}
We consider an homogeneous polynomial $P$ satisfying the condition of
the theorem, its quasi-critical  points $\{c_1,-c_1,\dots,c_k,-c_k\}$ and
$R_1,\dots,R_k$ the corresponding contact radius (which means
that the contact polynomial $Q_i = P_i + R_i$ has a singularity in
$c_i$ and $-c_i$ for $1 \leq i \leq k$).

By absurd, let us assume that $P$ is not a linear combination of 
$R_1,\dots,R_k$. Let $P_1$ be the orthogonal projection of $P$ on
the vector space generated by $R_1,\dots,R_k$ and let $D = P_1 - P
\neq 0$.
We also consider the sphere $\mathcal S_P$, centered at $P$ of radius  
$\dist(P,\Delta)$. This is schematically represented in figure \ref{figcomblin}.

\begin{figure}
\begin{tikzpicture}[semithick] 

\draw(0,0) ++ (-12:4) arc (-12:192:4);
\draw[->](0,0) -- (0:4) node [pos=0.0,below] {$P$} node [pos=0.5,below] {$R_1$};
\draw[->](0,0) -- (180:4) node [pos=0.5,below] {$R_2$};
\draw[->](0,0) -- (90:5.5) node [pos=0.5,left] {$D$} node [pos=1,left] {$P_1$};

\draw plot [smooth] coordinates {(-12:4.2) (4,0) (15:4.1) (30:4.2)
  (45:4.2) (60:4.2) (75:4.25) (90:4.3) (105:4.25) (120:4.2) (135:4.2)
  (150:4.2) (165:4.1) (180:4.0) (192:4.2)};

\draw[->](0,0.5) node [left] {$P_t$} -- (30:4.0) node [left,xshift=-0.08cm,yshift=0.1cm]
     {$Q^s_t$} --  (30:4.2) node
     [right] {$Q_t$};

\draw(0,0.5) --  (30:4.0) -- (3.4641,0.5) -- (3.4641,0.0);
\draw(3.4641,0.5) -- (0,0.5);

\draw (3.7,0.9) node {$h_t$}; 

\draw (3.6,-0.75) node {$S_P$}; 
\draw (4.4,-0.75) node {$\Delta$}; 

\end{tikzpicture}
\caption{Figure for the proof of theorem \ref{comblin}}\label{figcomblin}
\end{figure}
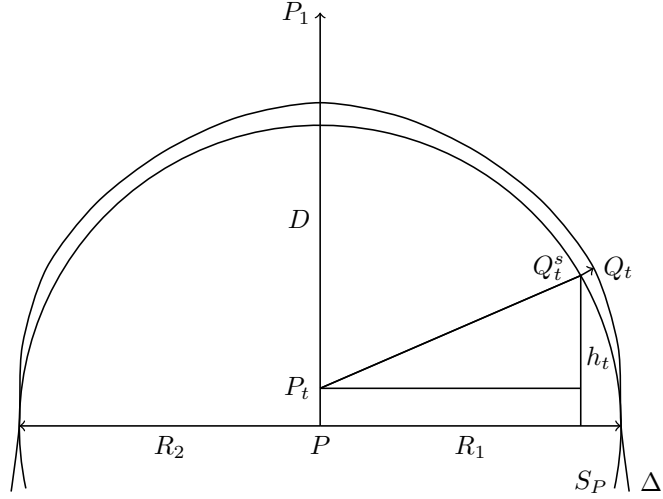

Next, we define $P_t = P + t D$. Let us choose a contact
polynomial $Q_t$ for $P_t$ which means that $\dist(P_t,\Delta) = 
\dist(P_t,Q_t)$. We consider $Q^s_t$, the intersection of 
the segment $[P_t, Q_t]$ with the sphere $\mathcal S_P$.
Let $h_t$ be the distance from $Q^s_t$ to the affine sub-space
containing $P$ and directed 
by $R_1,\dots,R_k$. We know that $\lim_{t \rightarrow 0} h_t = 0$
because $Q_t$ converges to the set  $\{P+R_1,\dots,P+R_k\}$.

We have~:
\begin{align*}
\dist^2(P_t,\Delta) &= \dist^2(P_t,Q_t) \\
&\geq \dist^2(P_t,Q^s_t) \\
&= (h_t - t \|D\|)^2 + \dist^2(P,\Delta) - h_t^2 \\
&= \dist^2(P,\Delta) -2 t h_t \|D\| + t^2\|D\|^2
\end{align*}

But, $P_t$ is not of norm $1$: $\|P_t\|^2 = \|P_1\|^2 + (1 -
t)^2\|D\|^2 = 1 - 2t\|D\|^2 + t^2\|D\|^2$ because $\|P_1\|^2 + \|D\|^2
= \|P\|^2 = 1$. Thus, we consider the
polynomial $\hat P_t = \frac{P_t}{\|P_t\|}$ and we have:

\begin{align*}
\dist^2(\hat P_t,\Delta) &= \frac{\dist^2(P_t,\Delta)}{1 - 2t\|D\|^2 + t^2\|D\|^2}
 \\
&\geq \frac{\dist^2(P,\Delta) -2 t h_t \|D\| + t^2\|D\|^2}{1 - 2t\|D\|^2 + t^2\|D\|^2}\\
&= (\dist^2(P,\Delta) -2 t h_t \|D\| + t^2\|D\|^2)(1 + 2t\|D\|^2 + o(t))\\
&= \dist^2(P,\Delta)(1 + 2t\|D\|^2) + o(t) \textrm{ because $t h_t \in o(t)$}
\end{align*}

This proves that $\dist^2(\hat P_t,\Delta) > \dist(P,\Delta)$ when $t$
is positive and small enough contradicting the fact that $P$ is a
local maxima for the distance to $\Delta$.
\end{proof}

Remarque: This gives a descent direction for an algorithm to compute 
local maxima for the distance to $\Delta$ that we use in the
experiments related in section \ref{experiments}.



\begin{corollary}\label{comblinbis}
Let $P$ be an homogeneous polynomial in $n$ variables, of
degree $d$, Bombieri norm $1$ and such that
$\dist(P,\Delta)$ is locally maximal among polynomials of Bombieri
norm $1$. Assume also that $P=0$ defines a locally extremal hyper-surface.

Let $\{c_1,-c_1,\dots,c_k,-c_k\}$
be the set of quasi-double points of $P$ (which are the
critical points of $P$ on $\mathcal S^{n-1}$ corresponding to the
smallest critical value in absolute value).

Then, we can find $\lambda_1, \dots, \lambda_k$ such that:
$$
P(x) = \sum_{i=1}^k \lambda_i \langle x | c_i \rangle^d
$$
\end{corollary}

\begin{proof}
We established that $P$ is a linear combination of $\{R_1,\dots,R_k\}$. 
By theorem \ref{extremal}, $P$ admits no
quasi-cusp point and therefore $R_i(x) = - P(c_i) \langle c_i | x
\rangle^d$ with $|P(c_i)| = \dist(P,\Delta)$.
\end{proof}


\section{The univariate case}

In this section, we assume $n = 2$, that is we consider homogeneous
polynomials of degree $d \geq 2$ with 
$2$ variables, which corresponds to univariate inhomogeneous polynomials.

For this section, it is simpler to manipulate trigonometric
polynomials in one variable. Therefore, to an homogeneous polynomial
$T$, we associate the function $\breve T : \RR \rightarrow \RR$
defined as
$\breve T(\theta) = T(u_\theta)$ with $u_\theta = (\cos(\theta), \sin(\theta))$.

It is clear that $T \mapsto \breve T$ is one to one and we can
therefore extend the Bombieri norm and scalar product to univariate
trigonometric polynomials of degree $d$ as $\|\breve T\| = \| T\|$.

Simple calculation shows that $\breve T'(\theta) = \langle
\nabla^T(u_\theta) | u_\theta^\bot \rangle$ and
$\breve T''(\theta) = {}^t u_\theta^\bot \mathcal H^T T(u_\theta) u_\theta^\bot
- d T(u_\theta)$.

We may think that the polynomial of degree $d$ with $2r$ roots
on the unit circle that maximizes
the distance to the discriminant, among polynomials of the same norm 
and number of roots, are likely to be the polynomials with regularly spaced
roots and having only two opposite critical values.

This leads to the polynomials $T$ of degree $d$, satisfying $\breve T(\theta) = \cos(r (\theta
+ \varphi))$. This
gives for $\varphi = 0$:

\begin{align*}
T_{r,d}(x,y) &= (x^2 + y^2)^\frac{d - r}{2} \sum_{k=0}^{\lfloor \frac{r}{2} \rfloor} (-1)^k
\binom{r}{2k} y^{2k} x^{r-2k} \\
&=  \sum_{p=0}^{\lfloor \frac{d}{2} \rfloor}  \left(\sum_{k=\max(0,p+\frac{r-d}{2})}^{\min(p,\frac{r}{2})} (-1)^k\binom{r}{2k}\binom{d-r}{2(p-k)}\right)y^{2p} x^{d - 2p}
\end{align*}

We can give a simple expression for the Bombieri norm of this polynomial when
$r=d$:

\begin{align*}
\|T_{d,d}\|^2 &= \sum_{k=0}^{\lfloor \frac{n}{2} \rfloor} \binom{d}{2k} =
2^{d-1} \\
\end{align*}

\begin{fact}
$\dist(T_{r,d},\Delta) = \min(1,\frac{r}{\sqrt{d}})$ 
\end{fact}

\begin{proof}
By theorem \ref{bombdist}, we have:
$\dist^2(T_{r,d},\Delta) = \min_{\theta \in [0, 2\pi[} \delta(\theta)$
with $\delta(\theta) = \breve T_{r,d}^2(\theta) + \frac{{\breve
    T'^2_{r,d}}(\theta)}{d} = \cos^2(r \theta) + \frac{r^2 \sin^2(r
  \theta)}{d}$.

We have $\delta'(\theta) = 2 r \frac{r^2 - d}{d} \cos(r \theta)\sin(r
\theta)$ hence, $\delta$ has critical values when either $\cos(r
\theta) = 0$ or $\sin(r
\theta) = 0$ which gives the result.
\end{proof}

Thus, if $r < \sqrt{d}$, then $\dist(T_{r,d},\Delta) < 1$ because the
polynomial has quasi-cusp points.
More generally, the proof of proposùition \ref{local} below suggests that we have
$\frac{d - r}{2}$
degrees of liberty to move $T_{r,d}$ away from the discriminant. When 
$r \leq \frac{d}{2}$, we may have enough degree of liberty to increase
the distance by changing the $r$ critical values.

Therefore, we only state the following conjecture:
\begin{conjecture}
Let $P$ an homogeneous polynomial in two variables
with $2r > d$ roots on the unit circle. Let $\alpha =
\dist(P,\Delta)$, we have:
$$
\alpha \leq \frac{\|T_{n,d}\|}{\|P\|}
$$
\end{conjecture}

If this conjecture is true, from the corollary \ref{comblinbis} and
and the fact that $T_{r,d}$ has no quasi-cusp points, we deduce that $\cos(r
\theta)$ should be a linear combination of the family $c_k(\theta) = cos^d(\theta -
\theta_k)$ where $\theta_k = \frac{k\pi}{r}$ are the extrema of
$\cos(r \theta)$ on the upper half of $\mathcal S^1$\!.

This is indeed true and we have the following (new ?) trigonometric
identities, which implies that the conjecture does not hold if $r < \frac{d}{3}$:

\begin{proposition}\label{trigo}
The following identities are true for any positive integer $d$:
\begin{align*}
 \cos(d \theta) &= \frac{2^{d-1}}{d} \sum_{k=0}^{d-1} (-1)^k \cos^d(\theta -
 \frac{k\pi}{d})  \\
 \sin(d \theta) &= \frac{2^{d-1}}{d} \sum_{k=0}^{d-1} (-1)^k \cos^d(\theta-
 \frac{2k+ 1}{2d}\pi)\end{align*}
and if $\frac{d}{3} < r \leq d$ with  $d - r$  even:
\begin{align*}
 \cos(r \theta) &= \frac{2^{d-1}}{r \binom{d}{\frac{d-r}{2}}} \sum_{k=0}^{r-1} (-1)^k \cos^d(\theta -
 \frac{k\pi}{r})  \\
 \sin(r \theta) &= \frac{2^{d-1}}{r \binom{d}{\frac{d-r}{2}}} \sum_{k=0}^{r-1} (-1)^k \cos^d(\theta-
 \frac{2k+ 1}{2r}\pi)  \\
\end{align*}
\end{proposition}

\begin{proof}
The first and second identities are particular cases of the third and
fourth when $r = d$.
We use the following reasonning for the third identity:
\begin{align*}
2^d \sum_{k=0}^{r-1} (-1)^k \cos^d(\theta -  \frac{k\pi}{r}) &= 
  \sum_{k=0}^{r-1} (-1)^k \left(e^{\mathrm i (\theta - \frac{k\pi}{r}) }
  + e^{-\mathrm i (\theta - \frac{k\pi}{r}) }\right)^d \\
  &=  \sum_{k=0}^{r-1} (-1)^k \sum_{p=0}^{d} \binom{d}{p} e^{\mathrm i
    p (\theta -  \frac{k\pi}{r})} e^{\mathrm i
    (p - d) (\theta - \frac{k\pi}{r})}\\
  &= \sum_{p=0}^{d}  \binom{d}{p} \sum_{k=0}^{r-1} (-1)^k e^{\mathrm i
    (2 p - d) (\theta - \frac{k\pi}{r})}\\
  &= \sum_{p=0}^{d}  \binom{d}{p} e^{\mathrm i
    (2 p - d)\theta}  \sum_{k=0}^{r-1} (-1)^k e^{\mathrm i
    (d - 2 p) \frac{k\pi}{r}} \\
  &= \sum_{p=0}^{d}  \binom{d}{p} e^{\mathrm i
    (2 p - d)\theta} \sum_{k=0}^{r-1} e^{\mathrm - i
    \frac{d - 2p - r}{r}k\pi} \\
  \hbox{The inner sum is non null only for } &2p \equiv d - r
  \pmod{2r}\hbox{, i.e. } 2 p = d \pm r\hbox{ where it is $r$.} \\
  &= r \binom{d}{\frac{d-r}{2}}  (e^{\mathrm i
    r \theta} + e^{\mathrm - i
    r \theta}) = 2 r
  \binom{d}{\frac{d-r}{2}}  \cos(r \theta) 
\end{align*}
The last identity is a consequence of the third one:
\begin{align*}
\sin(r \theta) = \cos\left(r\theta - \frac{\pi}{2}\right) 
&=  \frac{2^{d-1}}{r \binom{d}{\frac{d-r}{2}}}\sum_{k=0}^{r-1} (-1)^k \cos^d\left(\theta - \frac{\pi}{2r} -
 \frac{k\pi}{r}\right) \\
&=  \frac{2^{d-1}}{r \binom{d}{\frac{d-r}{2}}} \sum_{k=0}^{r-1} (-1)^k \cos^d\left(\theta -
 \frac{2k + 1}{2 r}\pi\right) 
\end{align*}
\end{proof}

The condition that $r$ is not too small is necessary, not only because
of the appearance of quasi-cusp points. 
Let us consider a polynomial of degree $5$ with $2$ roots on the unit
circle. By the previous fact, we know that $\dist^2(P,\Delta) =
\frac{1}{5}$ and that 
$T_{1,5}(x,y)$ only has two quasi-cusp points 
$(0,1)$ and $(0,-1)$ when $\sin(\theta) = 0$.
This means that  If it were a maximum of the distance to 
the discriminant, we would therefore have, using theorem \ref{comblin},
$T_{1,5}(x,y) = K y^4 x$ for some $K \in \RR$, which is not the case
because $T_{1,5}(x,y)= x(x^2+y^2)^2 = x^5 + 2x^3y^2 + xy^4$.
The same computation applies to $T_{2,6}$.

We were not able to prove that $T_{r,d}$ is a global maximum of the
distance to the discriminant when $d < 2r \leq 2d$. We were only able to prove that
$T_{d,d}$ is a local maximum:

\begin{proposition}\label{local}
$T_{d,d}$ is a local maximum of the distance to the discriminant among
  polynomials of the same norm.
\end{proposition}

\begin{proof}
The first thing to remark is that the polynomials of degree $d$ such that
$\breve T(\theta) = \cos(d (\theta
+ \varphi))$ for some $\varphi \in \RR$, are generated by the two 
polynomials $C_d$ and $S_d$ verifying  $\breve C_d(\theta) = \cos(d \theta)$ 
and $\breve S_d(\theta) = \sin(d \theta)$ ($C_d = T_{d,d}$). Moreover, 
it is easy to see that $S_d$ and $C_d$ are orthogonal for the Bombieri
scalar product.

If a polynomial $P$ is in the affine space generated by $C_d$ and $S_d$,
then we have $\dist(\frac{P}{\|P\|},\Delta) =
\dist(\frac{C}{\|C\|},\Delta)$ using the invariance of the Bombieri
norm.

Let us now consider a polynomial not in the affine space generated by
$C_d$ and $S_d$ and having the same norm as $S_d$ and $C_d$. Composing
$P$ with a rotation, we can assume that $P =
\alpha Q + \beta C_d$ with 
$\langle Q | C_d \rangle = 0$, $\langle Q | S_d \rangle = 0$,
$\alpha^2 + \beta^2 = 1$, $\alpha,\beta \geq 0$ and $\|Q\| = \|S_d\| = \|C_d\|$.
We also define $U_k(x) = \langle x |
u_{\frac{k\pi}{d}} \rangle^d$ with
$u_{\frac{k\pi}{d}}
= (\cos(\frac{k\pi}{d}),\sin(\frac{k\pi}{d})$). By proposition \ref{trigo},
we have
$$
C_d =  \frac{2^{d-1}}{d} \sum_{k=0}^{d-1} (-1)^k  U_k 
$$

This gives:
\begin{align*}
0 &= \langle Q | C_d \rangle \\
  &=  \frac{2^{d-1}}{d} \sum_{k=0}^{d-1} (-1)^k \langle Q | U_k \rangle \\
   &= \frac{2^{d-1}}{d} \sum_{k=0}^{d-1} (-1)^k Q(u_{\frac{k\pi}{d}}) \hbox{ using
corollary } \ref{veronese}
\end{align*}
We have $Q(u_{\frac{k\pi}{d}}) \neq 0$ for some $k$, otherwise,
$Q$ would be orthogonal to $S_d$ and all $U_k$ which implies $Q = 0$
since $\{S_d,U_1,\dots,U_k\}$ generates all polynomials.
Indeed, the $U_k$ are independant from proposition \ref{determinant} and orthogonal to $S_d$ using
corollary \ref{veronese} that gives $\langle S_d | U_k \rangle =
S_d(u_{\frac{k\pi}{d}}) = \sin(\frac{k\pi}{d}) = 0$.

This implies that there exists $k$ such that   $(-1)^k Q(u_{\frac{k\pi}{d}}) < 0$.
We have $C_d(u_{\frac{k\pi}{d}}) = \cos(\frac{k\pi}{d}) = (-1)^k$
which has an opposite sign to $Q(u_{\frac{k\pi}{d}})$ for such a $k$.  Hence,
using $C_d(u_{\frac{k\pi}{d}}) = (-1)^k$, $\breve C_d' = d \breve S_d$ and $S_d(u_{\frac{k\pi}{d}}) = 0$, 
\begin{align*}
\dist^2(P,\Delta) &\leq \dist^2(P,\Delta_{u_{\frac{k\pi}{d}}}) \\
&=\breve P^2(\textstyle \frac{k\pi}{d}) + \frac{1}{d} \breve
P'^2(\frac{k\pi}{d}) \\
&= \beta^2 + 2 \alpha \beta(-1)^k  \breve Q(\textstyle \frac{k\pi}{d}) + \alpha^2
(\breve Q^2(\frac{k\pi}{d}) + \frac{1}{d}
\breve Q'^2(\frac{k\pi}{d}))
\end{align*}
We may compute directly $\alpha^2 = 1 - \beta^2$ and $\beta^2 =  \frac{\langle P | C_d \rangle^2}{\|P\|^2\|C_d\|^2} + \frac{\langle P | S_d \rangle^2}{\|P\|^2\|S_d\|^2}$.
Hence, from $\alpha\beta(-1)^k  Q(u_k) < 0$ for some $k$, we deduce that if $P$ is near enough to the affine space generated 
 $C_d$ and $S_d$ with the same norm as those, then  $\dist(P,\Delta) < \dist(C_d,\Delta)$.
\end{proof}


\section{Critical band of extremal hyper-surfaces}

\begin{corollary}
Let $P \in \mathbb E$ be an homogeneous polynomial of degree $d$ with
$n$ variables. Assume that the zero level of $P$ on the unit sphere is locally
extremal. Let $m = \dist(P,\Delta)$,
$$\hbox{if } \|x\| = 1 \hbox{ and } |P(x)| \leq m \hbox{ then } 
  \|\nabla^T P(x)\|^2 > d (m^2 - P(x)^2)$$
\end{corollary}

\begin{proof}
By theorem \ref{bombdist}, we have for all $x$ in the unit sphere:

$$\dist(P,\Delta_x) = P(x)^2 + \frac{1}{d} \|\nabla^T P(x)\|^2 \geq
\dist(P,\Delta) = m$$

If it existed $x \in \mathcal S^{n-1}$ such that  $P(x)^2 + \frac{1}{d} \|\nabla^T P(x)\|^2
= \dist(P,\Delta)$, $x$ would be a
quasi-cusp, by definition, contradicting the first item of the previous theorem.

Thus, for all $x \in  \mathcal S^{n-1}$ we have $P(x)^2 + \frac{1}{d} \|\nabla^T P(x)\|^2 > m$ which
yields the wanted inequality.
\end{proof}

\begin{definition}
Let $P \in \mathbb E$ be an homogeneous polynomial of degree $d$ with
$n$ variables. Let $m = \dist(P,\Delta)$, the critical band of $P$ is
the following set:

$$\mathcal B_c(P) = \{ x \in \mathcal S^{n-1} \hbox{ s.t. } |P(x)| < m \}$$
\end{definition}

\begin{theorem}\label{bandwidth}
Let $P \in \mathbb E$ be an homogeneous polynomial of degree $d$ with
$n$ variables. Assume that the zero level of $P$ on the unit sphere is locally
extremal.

Let $\gamma : [a,b] \rightarrow \mathcal B_c(P)$
be an integral curve of $\nabla^T P$ with $a,b \in ]-m, m[$. Then, we have the following
inequality:
$$
\length(\gamma) < \frac{1}{\sqrt{d}} \left|\arcsin\left(\frac{b}{m}\right) - arcsin\left(\frac{a}{m}\right)\right|
$$
This is bounded by $\frac{\pi}{\sqrt{d}}$ and, if $a$ and $b$ have the
same sign, by half of it.
\end{theorem}

\begin{proof}
We can consider that $\gamma$ is parametrised by the value of $P$
because  $\nabla^T P$ does not vanish in $\mathcal B_c(P)$ and therefore,
the value of $P$ will be monotonous along the arc. We may also assume
without loss of generality that $a < b$.

This means that we have $P(\gamma(y)) = y$ which by derivation gives

$$
\langle \nabla P(\gamma(y)) | \gamma'(y) \rangle = \langle \nabla^T P(\gamma(y)) | \gamma'(y) \rangle = 1
$$ 

The previous corollary and the fact that $\gamma'(y)$ and $\nabla^T
P(\gamma(y))$ are colinear yields:

$$\|\gamma'(y)\| = \frac{1}{\|\nabla^T P(\gamma(y))\|} <
\frac{1}{\sqrt{d (m^2 - P(\gamma(y))^2)}} = \frac{1}{\sqrt{d (m^2 - y^2)}}$$

Finally, for the length, we have:

\begin{align*}
\hbox{length}(\gamma) &= \int_a^b \|\gamma'(y)\| \hbox{d}y\\
&<
\frac{1}{\sqrt{d}} \int_a^b \frac{\hbox{d}y}{\sqrt{m^2 - y^2}}\\
&< \frac{1}{\sqrt{d}}  \left(\arcsin\left(\frac{b}{m}\right) 
- arcsin\left(\frac{a}{m}\right)\right)
\end{align*}

If $b < a$, reversing the arc gives the wanted result.
\end{proof}





\section{Large components far from the discriminant}

The following proposition will have as consequence a lower bound for the size 
of connected components of an algebraic hyper-surfaces:
\begin{proposition}
Let $P \in \mathbb E$ (i.e. $P$ is an homogeneous
polynomial of degree $d$ with $n$
variables). Let $\dist(P,\Delta) \neq 0$ be the
distance between $P$ and the real discriminant of $\mathbb E$, then 
for any $x \in \mathcal S^{n-1}$ a critical point of $P$, the open spherical cap of $\mathcal
S^{n-1}$ with center
$x$ and radius angle 
$\alpha = \frac{1}{d}\sqrt{\frac{2 \dist(P,\Delta)}{\|P\|}}$ does
  not meet the zero level of $P$.
\end{proposition}

\begin{proof}
Let us consider $x \in \mathcal S^{n-1}$ a critical point of $P$ and
$y \in \mathcal S^{n-1}$ such that $P(y) = 0$.
Consider $\alpha$ the measure of the angle $x0y$, and consider $v \in
\mathcal S^{n-1}$ such that $v$ orthogonal to $x$ and $y =
\cos(\alpha)x + \sin(\alpha)v$. See figure \ref{figone}.

\begin{figure}
\begin{tikzpicture}[semithick]
\def\radius{3}
\draw (0,0) ++ (-15:\radius) arc (-15:105:\radius);
\node at  (105:\radius) [above right] {$\mathcal S^{n-1}$};
\draw (0,0) -- (\radius,0) node [pos=0.6,above] {$x$};
\draw (0,0) -- (0,\radius) node [pos=0.6,left] {$v$};
\draw (0,0) -- (35:\radius) node [pos=0.6,above left] {$y$};
\draw (0,0) ++ (0:0.2*\radius) arc (0:35:0.2*\radius);

\draw[->] (\radius,0) -- (1.5*\radius,0) node [pos=0.5,above]
     {$\nabla P(x)$};
\draw[->] (0,0) ++ (35:\radius) -- ++ (125:\radius*0.5) node
     [pos=0.5,above right]
     {$\nabla P(y)$};

\node at (17.5:0.3*\radius) {$\alpha$};
\node at (0,0) [left] {$0$};
\end{tikzpicture} 
\caption{}\label{figone}
\end{figure} 

Then, we define:
$$
f(\theta) = P(\cos(\theta)x + \sin(\theta)v)
$$
We have:
$$
\begin{array}{rcl}
  f(0) &=& P(x) \cr
  f(\alpha) &=& P(y) = 0 \cr
  f'(\theta) &=& {^t(-\sin(\theta)x + \cos(\theta)v)}\nabla
  P(\cos(\theta)x + \sin(\theta)v) \cr
  f'(0) &=& {^t v}\nabla P(x) = 0 \hbox{ because $v$ orthogonal to $x$
  and $\nabla P(x)$} \cr
  f''(\theta) &=&  {^t(-\sin(\theta)x + \cos(\theta)v)}\mathcal H
  P(\cos(\theta)x + \sin(\theta)v)(-\sin(\theta)x + \cos(\theta)v) \cr
  &&+\;
  {^t(-\cos(\theta)x - \sin(\theta)v)}\nabla
  P(\cos(\theta)x + \sin(\theta)v)
\end{array}
$$
Using the inequality of lemma \ref{bombineq} and the fact that
$\cos(\theta)x + sin(\theta)v \in \mathcal S^{n-1}$, we have:
$$
\begin{array}{rcl}
|f''(\theta)| &\leq& \|\mathcal H P(\cos(\theta)x +
\sin(\theta)v)\|_2 \, \|-\sin(\theta)x + \cos(\theta)v\|^2 \cr &&+\;
\|\nabla P(\cos(\theta)x + \sin(\theta)v)\| \, \|-\cos(\theta)x -
\sin(\theta)v\| \cr
&=& \|\mathcal H P(\cos(\theta)x +
\sin(\theta)v)\|_2 +
\|\nabla P(\cos(\theta)x + \sin(\theta)v)\| \cr
&\leq& d(d-1)\, \|P\| \,\|\cos(\theta)x +\sin(\theta)v)\|^{d-2} + 
d\, \|P\| \,\|\cos(\theta)x +\sin(\theta)v)\|^{d-1} \cr
&=& d^2 \|P\|
\end{array}
$$
Then, using Taylor-Lagrange equality, we find $\theta \in [0,\alpha]$ such
that
$$
0 = f(\alpha) = f(0) + \alpha f'(0) + \frac{\alpha^2}{2}f''(\theta) = 
P(x)  + \frac{\alpha^2}{2}f''(\theta)
$$

This implies:
$$
|P(x)| \leq  \frac{d^2 \alpha^2}{2} \|P\|
$$

and therefore with theorem \ref{bombineqdist}, we have
$$
\alpha \geq \frac{1}{d}\sqrt{\frac{2 \,\dist(P,\Delta)}{\|P\|}}
$$
\end{proof}

\begin{corollary}\label{sphereinside}
Let $P \in \mathbb E$ and $\dist(P,\Delta) \neq 0$ be the
distance between $P$ and the discriminant for $\mathbb E$.
Each connected component of the complement of the zero level of $P$
in $\mathcal
S^{n-1}$ contains an open spherical cap of $\mathcal
S^{n-1}$ with center
$x$ and radius angle 
$\alpha = \frac{1}{d}\sqrt{\frac{2 \dist(P,\Delta)}{\|P\|}}$
\end{corollary} 

\begin{proof}
Immediate because every connected component of the complement of
the zero level of $P$ contains at least one extrema of $P$ which is a
critical point of $P$.
\end{proof}

These two last results can also be used in the projective space
$\mathcal P^{n-1}(\mathbb R)$ with the metric induced by the metric on the
sphere $\mathcal S^{n-1}$ because the radius angle of the spherical
cap
in  $\mathcal S^{n-1}$
is the radius of a disk is  $\mathcal P^{n-1}(\mathbb{R})$. 

\begin{theorem}\label{distancebetween}
Let $P \in \mathbb S$ and $\dist(P,\Delta) \neq 0$ be the
distance between $P$ and the discriminant for $\mathbb E$.
The distance (measured as an arc length) between two distinct connected components 
of the zero level of $P$
in $\mathcal
S^{n-1}$ is greater of equal to 
$\alpha = \frac{2}{d}\sqrt{\frac{2 \dist(P,\Delta)}{\|P\|}}$
\end{theorem}

\begin{proof}
Consider an arc $[A,B]$ on  $\mathcal S^{n-1}$ joining two distinct
connected components of the zero level of $P$. By the Ehresmann
theorem \cite{Ehre50}, There is a point $C$  on $[A,B]$ where $P$ reaches a value
greater, in absolute value, than a critical value
of $P$. We can take $C$ a point where $|P(x)|$ is maximum on
$[A,B]$.

Then, by exactly the same computation than for the previous theorem,
using the fact that $\nabla P(C)$ is zero in the direction of the
segment, we find that the arc lengths of $[A,C]$ and $[C,B]$ are
greater or equal to $\frac{1}{d}\sqrt{\frac{2 \dist(P,\Delta)}{\|P\|}}$ which ends
the proof.

\end{proof}


\section{Experiments with extremal curves}\label{experiments}

Section \ref{further} suggests an algorithm to numerically optimise the
distance to the discriminant of an hyper-surface: this algorithm 
requires to compute the 
quasi-singular points and solve a linear system to get a direction in
which the distance increases. At each step, we do not need to recompute
the quasi-singular points because we can use Newton's method to move
the previous ones.

We have implemented such an algorithm as part of our GlSurf software. It is not a robust algorithm
and it sometimes encounter numerical problems. Nevertheless, we
managed to use it on all maximal curves up to degree 6 (inclusive) and
some curves of higher degree. We relate those experiments in the table
\ref{tableexperiments} in the hope that they could help to build new conjectures.

Remark: because we deal with curves, we reinforced our theorem taking
into account the nesting of connected components which give more
locally extremal hyper-surfaces that just considering the number of
connected components: there are two sextic curves with nine ovals that
do not have the maximum $b_0 = 11$ but that are locally extremal. 

\begin{table}
\begin{tabular}{|c|c|c|c|c|}
\hline
Degree & Topology & $\dist(P,\Delta)$ & $\|D\|^2$ & $k = k^+ + k^-$ \\
\hline
2 & $O$ & $\frac{1}{\sqrt{3}} \simeq 0.577 $ & $0$ & $\infty = 1 + \infty$ \\
3 & $| \; O$ & $\frac{1}{\sqrt{7}} \simeq 0.378 $ & $0$ & $\infty = 1 + \infty$ \\
4 & $O\times4$ & $7.12 10^{-2}$ & $2.03 10^{-35}$ & 10 = 4 +
6\\
4 & $(O)$ & $\frac{\sqrt{3}}{\sqrt{47}} \simeq 0.253$ & $0$ & $\infty = \infty + \infty$ \\
5 & $| \; O\times6$ & $2.49 10^{-3}$ & $4.27 10^{-19}$ & $15 = 6 +
9$\\
5 & $| \; (O)$ & $\frac{\sqrt{3}}{\sqrt{103}} \simeq 0.171$ & $0$ & $\infty = \infty + \infty$ \\
6 & $(O) \; O\times9$ & $9.61 10^{-5}$ & $1.30 10^{-28}$ & $22 = 9 + 13$  \\
6 & $(O \times 5) \; O\times5$ & $6.79 10^{-11}$ & $\bf 1.64
  10^{-1}$ & $21 = 10 + 11$ \\
6 & $(O \times 9) \; O$ & $5.82 10^{-9}$ & $\bf 5.05 10^{-4}$ & $21 =
11 + 10$ \\
6 & $(O \times 2) \; O \times 6$ & $2.49 10^{-4}$ & $3.14 10^{-13}$ & $21 = 10
+ 11 $ \\
6 & $(O \times 6) \; O \times 2$ & $6.56 10^{-7}$ & $3.25 10^{-15}$ & $20 = 10
+ 10 $ \\
6 & $((O))$ & $\frac{\sqrt{5}}{\sqrt{371}}$ & $0$ & $\infty = \infty + \infty$ \\
7 & $| \; O\times15$ & $2.38 10^{-6}$ & $5.54 10^{-12}$ & $30 = 15 + 15$\\
8 & $(O\times3) O\times 18$ & $3.43 10^{-8}$ & $\bf 1.51 10^{-6}$ & $39 = 18 + 21$ \\
\hline
\end{tabular}
\begin{description}
\item[Degree] the degree of the polynomial (which has Bombieri norm 1)
\item[Topology] $O$ represents an oval, $|$ a projective lines 
$(...)$ an oval containing other curves and we use $\times$ to shorten
  the representation.
\item[$\bf{dist(P,\Delta)}$] the distance to the discriminant at the stage
we stopped the optimisation.
\item[$\bf{ \left\|D\right\| }$] following the notation of the proof of theorem
  \ref{comblin}. It converges toward $0$ and gives a good indication
to know it we are near the local minima. We highlighted in bold 
values which are not small enough to draw definitive conclusion.
We choosed $\|D\|^2 < \dist^2(P,\Delta)$ which ensures a correct
topology for the linear combination given by corollary \ref{comblinbis}.
\item[$\bf{k = k^+ + k^-}$] The number of pairs of quasi-double points together with the
  sign of the polynomials. For odd degree, this make only sense if we
  assume that the sign is taken on same half of
  $\mathcal{S}^{n-1}$, split by the projective line that is always
  present in $P=0$. We count in $k^+$ the critical values which are
  inside ovals not contained in other ovals.
\end{description}
\caption{Experimental results}\label{tableexperiments}
\end{table}

It is important to note that these experiments only produce 
polynomials 
of Bombieri norm $1$ that are likely to be near a local maxima of the
distance to the discriminant. We currently have no way to find
accurately (numerically or theoretically) the global maxima in each
connected component of the complement of the real discriminant.

\paragraph{\bf Degree 2, 3 and maximally nested ovals}

Curves which have the maximum number of nested ovals are a particular case (i.e. $d/2$ nested
ovals if the degree $d$ is even and $(d-1)/2$ nested ovals plus a
projective line otherwise).

It seems from experiments that the polynomial maximising the distance
to the discriminant in these cases is obtained as the revolution of a
polynomial in two variables that has the maximum number of roots
equally spaced on the unit circle. This may be defined as:

\begin{align*}
T_d(x,t) &= \sum_{k=0}^{\lfloor \frac{d}{2} \rfloor} (-1)^k
\binom{d}{2k} t^{k} x^{d-2k} \\
P_d(x,y,z) &= T_d(x,y^2+z^2)
\end{align*}

The polynomials $P_d$ are not of norm $1$, and have infinitely many critical
points with the critical value $\pm 1$. In fact the curve $P_d(1,x,y)
= 0$ is a union of concentric circles centred in $(1,0,0)$ plus one line at
infinity, when $d$ is odd. Therefore, they are not irreducible. The point $(1,0,0)$ is the
only isolated critical point. This is how we filled
the corresponding lines of the table \ref{tableexperiments} with
$\frac{1}{\|P_d\|}$. Experiments seems to indicate that these are
global maxima of the corresponding connected components of $\Delta^c$,
but proving this probably implies finishing to study the univariate case.

\paragraph{\bf Extremal sextic curves and beyond}

The table \ref{tableexperiments} includes five locally extremal
sextic curves including the three curves with eleven components.

The most {\em difficult} one is Gudkov's sextic followed by
Hilbert's. We have run our optimisation algorithm on these curves for
several months! Yet, more computation are needed, as shown in the
table. Computing accurately the direction of descent requires to use 
more than the 64 bits of precision available in {\em modern}
processors. It is a pity that 128 bits have been abandoned in hardware,
but luckily we used GNU MP.

Remark: those curves are most of the time (always as far as the
author knows) shown in
literature as schemata. We give here in figure \ref{figharnack}
to \ref{fighilbertter} drawing of the {\em real} curves. If you are
interested to get the corresponding polynomials, you are welcome to
visit the following web page:

\begin{center}
\verb#http://lama.univ-savoie.fr/~raffalli/glsurf-optimisation.php#
\end{center}

\begin{figure}
\includegraphics[scale=.2]{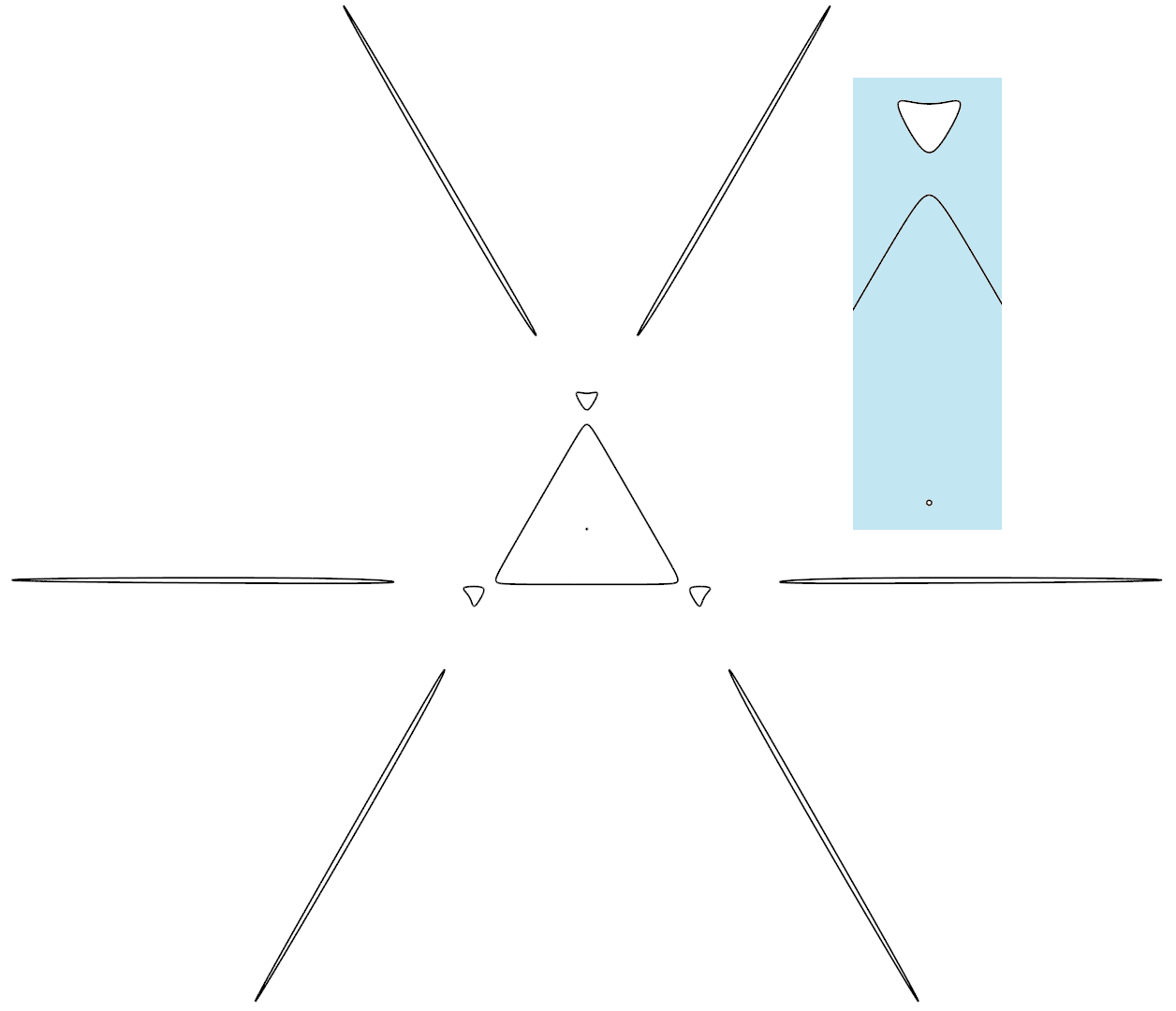}
\caption{Harnack's sextic}\label{figharnack}
\end{figure}

\begin{figure}
\includegraphics[scale=.2]{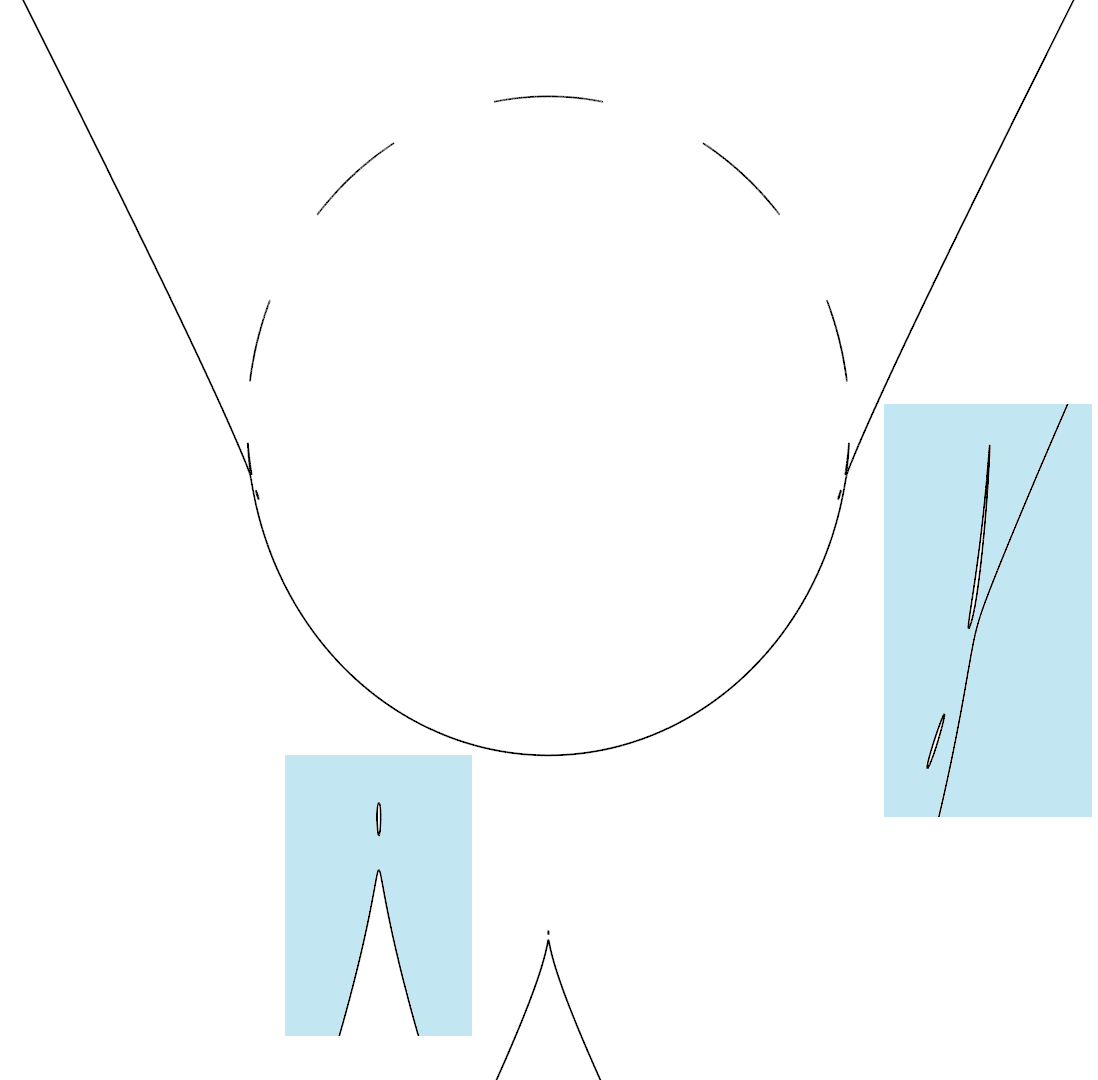}
\caption{Hilberts' sextic}\label{fighilbert}
\end{figure}

\begin{figure}
\includegraphics[scale=.2]{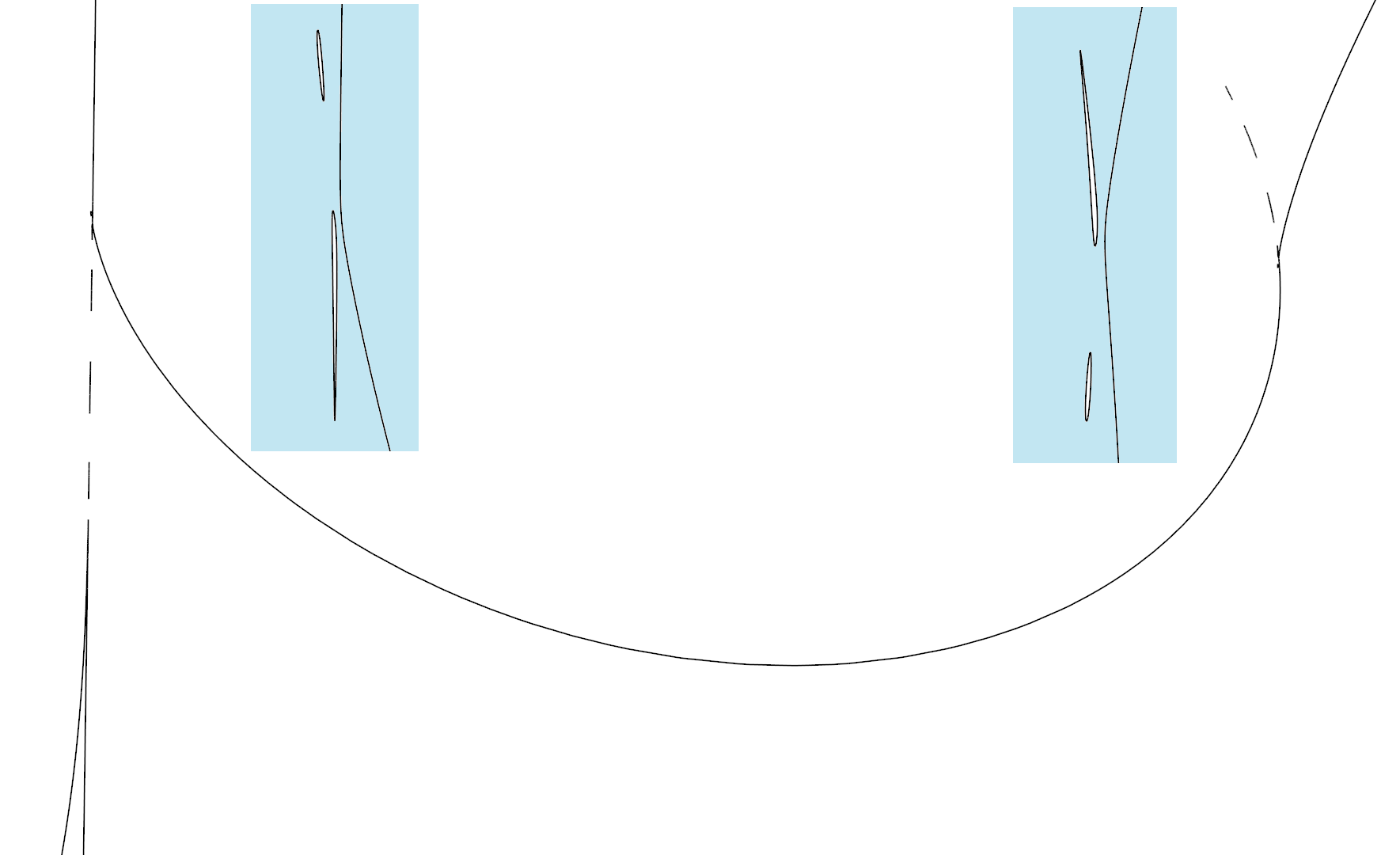}
\caption{Gudkov's sextic}\label{figgudgov}
\end{figure}

The extremal curves which have only nine components are interesting. They show some {\em useless} pikes which are quite
surprising. Probably, these pikes are necessary to have enough critical
points for corollary \ref{comblinbis} to hold.

\begin{figure}
\includegraphics[scale=.2]{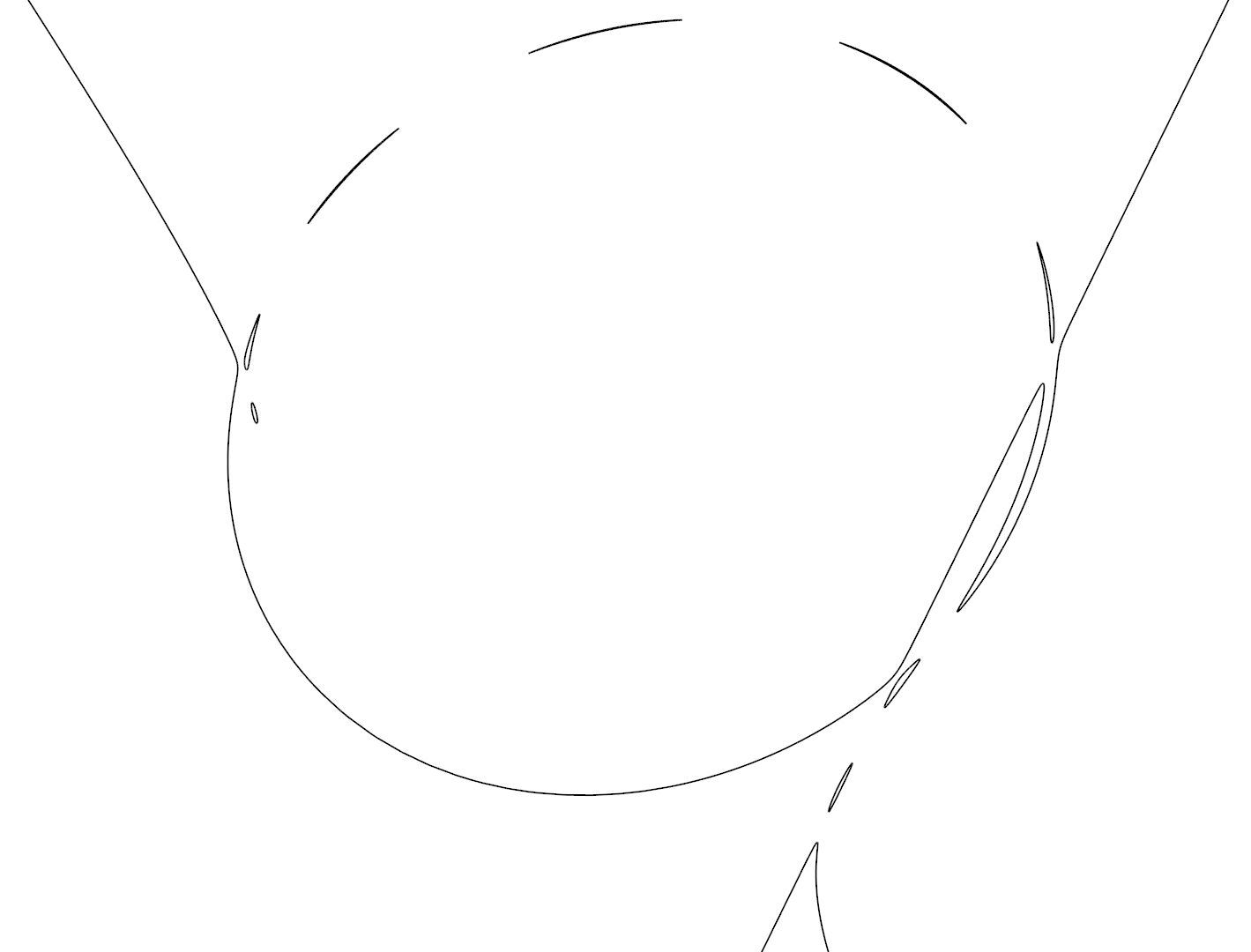}
\caption{Sextic with topology $(O\times 6) \;  OO$}\label{fighilbert62}\label{fighilbertbis}
\end{figure}

\begin{figure}
\includegraphics[scale=.2]{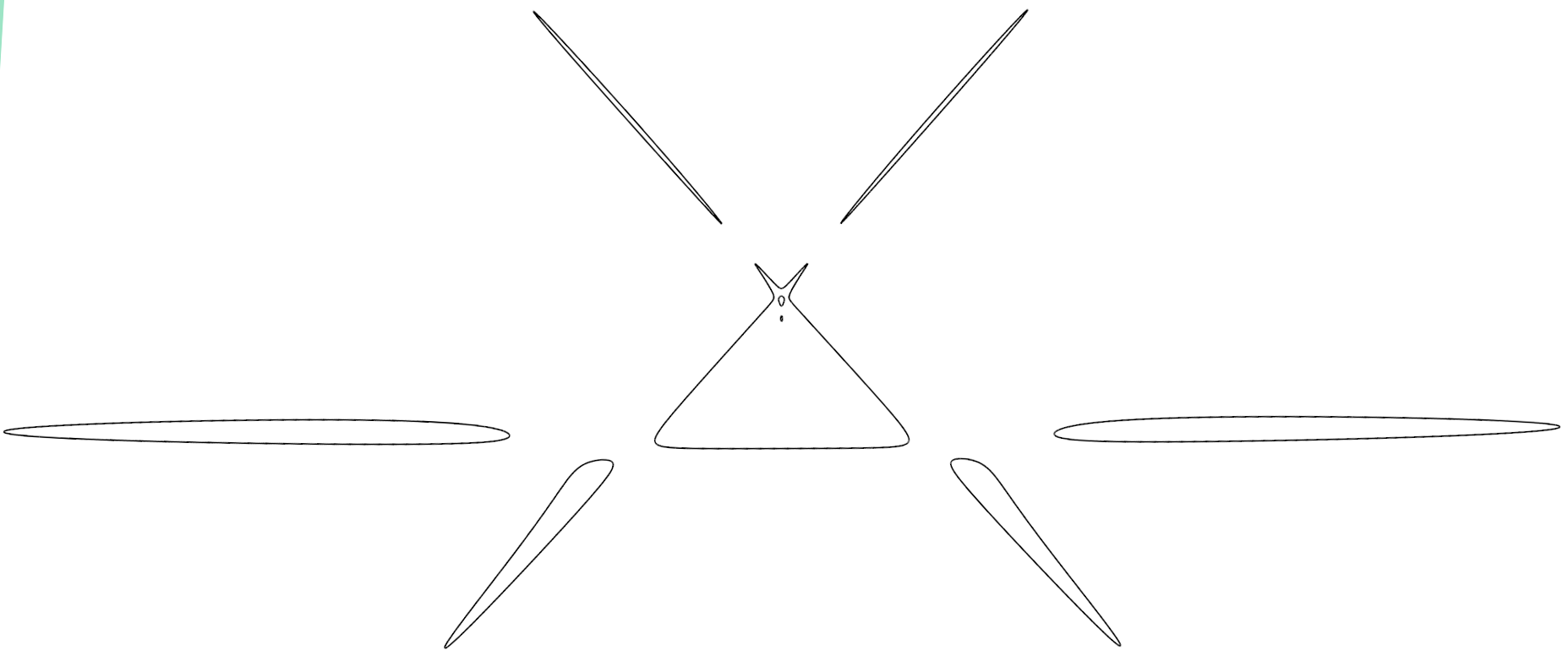}
\caption{Sextic with topology $(OO) \; O\times 6$}\label{fighilbertter} 
\end{figure}

We have also included results for Harnack's curves of degree seven and
height. The later is also not yet optimised enough.


\section{Conclusion}

There remains a lot of open problems in this work. Those we find the
most interesting are:

\begin{enumerate}
\item Complete the univariate case. We were surprised that 
proving that the polynomials $T_{r,d}$ are (or not) the global maximum to the
distance to the discriminant among polynomials with the same norm and
number of roots, is not easy, even when $r = d$.

\item In the general case, we could search for an upper bound to the number of
  terms in the identity given by corollary \ref{comblinbis}, from a bound for the number of critical values of a polynomials
  on two levels. For curves, such a bound is given by Chmutov
  in \cite{Chmutov95}; an asymptotic equivalent is $\frac{7}{8} d^2$. However, this result gives a bound which is greater
  than the dimension of the space of curves: $\simeq
  \frac{1}{2} d^2$ and we expect a better bound from our experimental results.

\item A lower bound for the same quantity seems much harder and could
  lead to proof that some topology can not be realised with a given
  degree ...

\item More generally, the points $ \{c_1,\dots,c_k\}$ on the sphere
  that are used  by the identity
$$
P(x) = \sum_{i=1}^k \lambda_i \langle x | c_i \rangle^d
$$
in  corollary \ref{comblinbis} are solution of a family of algebraic
systems.
If we know $k$ and the sign $s_i \in \{-1,1\}$ of $P$ at $c_i$ for $1
\leq i \leq k$, we have linear equations for the
$\lambda_i$
by writing $P(c_i) = s_i$. Then, writing that $c_i$ is a critical
point of $P$ completes the algebraic system.

Finding all solutions of these systems for all possible $k \in \NN$ and
$s_1,\dots,s_k \in \{-1,1\}$ and determining the topology of
the corresponding polynomials would mean solving Hilbert's 16th
problem about the topology of algebraic curves.

\end{enumerate}


\bibliography{biblio}
\bibliographystyle{plain}

\appendix

\section{Proof of the inequalities for the Bombieri norm}\label{appbomb}

We now prove the inequalities of lemma \ref{bombineq}:
\begin{align*}
|P(x)| &\leq \|P\| \|x\|^{d} \cr
\|\nabla P(x)\| &\leq d \, \|P\| \,\|x\|^{d-1} \cr
\|\mathcal{H} P(x)\|_2 \leq \|\mathcal{H} P(x)\|_F &\leq
\displaystyle d(d-1) \,\|P\|\, \|x\|^{d-2} \cr
\end{align*}

We consider that $P_{\mathcal B} =
(a_\alpha)_{|\alpha|=d}$ and therefore, 
$\displaystyle P(x) = \sum_{|\alpha|=d} a_\alpha \sqrt{\frac{d!}{\alpha!}} x^\alpha$:
\begin{enumerate}
\item For the first inequality, the proof is easy:
 $$
  \begin{array}{rcl}
    P(x)^2 &=& \displaystyle \left(\sum_{|\alpha|=d} a_\alpha \sqrt{\frac{d!}{\alpha!}} x^\alpha\right)^2 \cr
 &\leq& \displaystyle \sum_{|\alpha|=d} a_\alpha^2
 \sum_{|\alpha|=d}\frac{d!}{\alpha!} x^{2\alpha} \hbox{ by
   Cauchy-Schwartz inequality}\cr
 &=&  \displaystyle \|P\|^2   \|x\|^{2d}
  \end{array}$$
\item For the second inequality, we first consider the partial
  derivative $\frac{\partial P(x)}{\partial x_i}$: 
$$
  \begin{array}{rcl}
    \displaystyle\left( \frac{\partial P(x)}{\partial x_i}\right)^2 &=& \displaystyle \left(\sum_{|\alpha|=d} a_\alpha \sqrt{\frac{d!}{\alpha!}}\alpha_i x^{\alpha-\chi_i}\right)^2 \cr
 &\leq& \displaystyle \sum_{|\alpha|=d} \alpha_i a_\alpha^2
 \sum_{|\alpha|=d}\frac{d!}{\alpha!} \alpha_i x^{2(\alpha - \chi_i)}
 \hbox{ by Cauchy-Schwartz}\cr
 &=&  \displaystyle d \sum_{|\alpha|=d} \alpha_i a_\alpha^2
 \sum_{|\alpha|=d, \alpha_i \neq 0}\frac{(d-1)!}{(\alpha-\chi_i)!}  x^{2(\alpha - \chi_i)} \cr
 &=&  \displaystyle d \sum_{|\alpha|=d} \alpha_i a_\alpha^2
 \sum_{|\beta|=d-1}\frac{(d-1)!}{\beta!}  x^{2\beta} \cr
 &=&  \displaystyle d \|x\|^{2(d-1)} \sum_{|\alpha|=d} \alpha_i a_\alpha^2 
  \end{array}$$
This means that:
 $$
  \begin{array}{rcl}
\|\nabla P(x)\|^2 &=& \displaystyle\sum_{1 \leq i \leq n}  \left(
  \frac{\partial P(x)}{\partial x_i}\right)^2 \cr
&\leq&  \displaystyle\sum_{1 \leq i \leq n}  \left( d  \|x\|^{2(d-1)} \sum_{|\alpha|=d}
  \alpha_i a_\alpha^2 \right) \cr
&=&  \displaystyle d\|x\|^{2(d-1)} \sum_{1 \leq i \leq n}\sum_{|\alpha|=d}
  \alpha_i a_\alpha^2 \cr
&=&  \displaystyle d\|x\|^{2(d-1)} \sum_{|\alpha|=d}
  \left(\sum_{1 \leq i \leq n} \alpha_i\right) a_\alpha^2 \cr
&=&  \displaystyle d^2\|x\|^{2(d-1)} \sum_{|\alpha|=d} a_\alpha^2 \cr
&=&  \displaystyle d^2 \|P\|^2 \|x\|^{2(d-1)} 
  \end{array}$$
\item For the last inequality, we consider the partial
  derivative $\displaystyle \frac{\partial^2 P(x)}{\partial x_i x_j}$
  when $i \neq j$: 
$$
  \begin{array}{rcl}
    \displaystyle\left( \frac{\partial^2 P(x)}{\partial x_i
        x_j}\right)^2 &=& 
\displaystyle \left(\sum_{|\alpha|=d} a_\alpha \sqrt{\frac{d!}{\alpha!}}\alpha_i\alpha_j x^{\alpha-\chi_i-\chi_j}\right)^2 \cr
 &\leq& \displaystyle \sum_{|\alpha|=d} \alpha_i\alpha_j a_\alpha^2
 \sum_{|\alpha|=d}\frac{d!}{\alpha!} \alpha_i\alpha_j x^{2(\alpha - \chi_i-\chi_j)} \hbox{ by Cauchy-Schwartz}\cr
 &=&  \displaystyle d(d-1) \sum_{|\alpha|=d} \alpha_i\alpha_j a_\alpha^2
 \sum_{|\alpha|=d, \alpha_i \neq 0,\alpha_j \neq 0}\frac{(d-2)!}{\alpha-\chi_i-\chi_j!}  x^{2(\alpha - \chi_i-\chi_j)} \cr
 &=&  \displaystyle d(d-1) \sum_{|\alpha|=d} \alpha_i\alpha_j  a_\alpha^2
 \sum_{|\beta|=d-2}\frac{(d-2)!}{\beta!}  x^{2\beta} \cr
 &=&  \displaystyle d(d-1) \|x\|^{2(d-2)} \sum_{|\alpha|=d} \alpha_i\alpha_j  a_\alpha^2 
  \end{array}$$
Now, we consider the partial
  derivative $\frac{\partial^2 P(x)}{\partial x_i^2}$:
$$
  \begin{array}{rcl}
    \displaystyle\left( \frac{\partial^2 P(x)}{\partial x_i^2}\right)^2 &=& 
\displaystyle \left(\sum_{|\alpha|=d} a_\alpha \sqrt{\frac{d!}{\alpha!}}\alpha_i(\alpha_i-1) x^{\alpha-2\chi_i}\right)^2 \cr
 &\leq& \displaystyle \sum_{|\alpha|=d} \alpha_i(\alpha_i-1)
 a_\alpha^2 
 \sum_{|\alpha|=d}\frac{d!}{\alpha!} \alpha_i(\alpha_i-1) x^{2(\alpha - 2\chi_i)} \hbox{ by Cauchy-Schwartz}\cr
 &=&  \displaystyle d(d-1) \sum_{|\alpha|=d} \alpha_i(\alpha_i-1) a_\alpha^2
 \sum_{|\alpha|=d, \alpha_i \geq 2}\frac{(d-2)!}{(\alpha-2\chi_i)!}  x^{2(\alpha - 2\chi_i)} \cr
 &=&  \displaystyle d(d-1) \sum_{|\alpha|=d} \alpha_i(\alpha_i-1)  a_\alpha^2
 \sum_{|\beta|=d-2}\frac{(d-2)!}{\beta!}  x^{2\beta} \cr
 &=&  \displaystyle d(d-1) \|x\|^{2(d-2)} \sum_{|\alpha|=d} \alpha_i(\alpha_i-1)  a_\alpha^2 
  \end{array}$$

Let us define $\iota_{i,j} = 0$ when $i \neq j$ and  $\iota_{i,i} =
1$. Then, we have:
 $$
  \begin{array}{rcl}
\|\mathcal H P(x)\|_F^2 &=& \displaystyle\sum_{1 \leq i,j \leq n}  \left(
  \frac{\partial^2 P(x)}{x_i x_j}\right)^2 \cr
&\leq&  \displaystyle\sum_{1 \leq i,j \leq n}  \left( d(d-1)  \|x\|^{2(d-2)} \sum_{|\alpha|=d}
  \alpha_i(\alpha_j - \iota_{i,j}) a_\alpha^2 \right) \cr
&=&  \displaystyle d(d-1)\|x\|^{2(d-2)} \sum_{1 \leq i,j \leq n}\sum_{|\alpha|=d}
  \alpha_i(\alpha_j - \iota_{i,j}) a_\alpha^2 \cr
&=&  \displaystyle d(d-1)\|x\|^{2(d-2)} \sum_{1 \leq i \leq n}\sum_{|\alpha|=d}
  \alpha_i(d - 1) a_\alpha^2 \cr
&=&  \displaystyle d(d-1)\|x\|^{2(d-2)}\sum_{|\alpha|=d}
  d(d - 1) a_\alpha^2 \cr
&=&  \displaystyle d^2(d-1)^2 \|P\|^2 \|x\|^{2(d-2)} 
  \end{array}$$

\end{enumerate}

\section{Independance of  $U_k(x) = \langle x | u_k\rangle^d$}

We need the following lemma:
\begin{lemma}\label{determinant}
Let $\{u_0,\dots,u_d\}$ be distinct points in $\mathcal S^1$.
Let $U_k(x) = \langle x | u_k\rangle^d$ for $0 \leq k \leq d$.

Then, the family of polynomials
$\{U_0,\dots,U_d\}$ is linearly independant and therefore a base of 
the space of homogeneous polynomials of degree $d$ in $2$ variables.
\end{lemma}

\begin{proof}[Proof of the lemma]
We define $\theta_k \in [0,2\pi[$ such that $c_k = (\cos(\theta_k), \sin(\theta_k))$.

Using lemma \ref{scalarlinear}, we find that
\begin{align*}
\langle R_i | R_j\rangle &= \langle c_i | c_j\rangle^d\\ &=
\cos^d(\theta_i - \theta_j)
\end{align*}

We define the $(d+1) \times (d+1)$ symmetrical matrix which is
the Gramian matrix of the family $\{U_0,\dots,U_d\}$ with respect to
Bombieri scalar product:
$$G = \left(
\begin{matrix}
1 & \cos^d(\theta_0 - \theta_1) & \dots &
  \cos^d(\theta_0 - \theta_d) \\
\cos^d(\theta_1 - \theta_0) & 1 & \dots &
  \cos^d(\theta_1 - \theta_d)  \\
\vdots & \vdots & \ddots & \vdots \\
\cos^d(\theta_d - \theta_0) & \cos^d(\theta_d - \theta_1) &\dots & 1
\end{matrix}
\right)$$
We find that $G$ is the matrix with its
$(i,j)$ coefficient equal to
\begin{align*}
\cos^d(\theta_i - \theta_j) &= (\cos(\theta_i)\cos(\theta_j)
- \sin(\theta_i)\sin(\theta_j))^d \\
&= \sum_{k=0}^{d} \binom{d}{k} \cos^k(\theta_i)\sin^{d-k}(\theta_i)
\cos^k(\theta_j)\sin^{d-k}(\theta_j)
\end{align*}
Hence, we find that
$$G = {}^tV D V$$
where $D$ is the diagonal matrix with coefficient 
$(k,k)$ equals to $\binom{d}{k}$ and $V$
is a matrix with the $(k,i)$ coefficient equals to 
$\cos^k(\theta_i)\sin^{d-k}(\theta_i)$.

We remark that $V$ is an homogeneous Vandermonde matrix
whose determinant is $\prod_{0 \leq i < j \leq d} \sin(\theta_i -
\theta_j)$
which gives the wanted result.
\end{proof}


\end{document}